\documentclass{elsarticle}

\usepackage{lineno,hyperref}
\usepackage{amsmath}
\usepackage{amsfonts}
\usepackage{amssymb}
\usepackage{graphicx}
\usepackage{hyperref}
\usepackage{amsthm}
\usepackage{bookmark}
\usepackage{color}
\usepackage{xcolor}
\usepackage{caption}
\usepackage{multirow}
\usepackage{natbib}

\journal{Journal of Multivariate Analysis}

\newtheorem{theorem}{Theorem}[section]
\newtheorem*{theorem*}{Theorem 1.1}
\newtheorem*{acknow*}{Acknowledgement}

\newtheorem{example}[theorem]{Example}
\newtheorem{lemma}[theorem]{Lemma}

\newtheorem{proposition}[theorem]{Proposition}
\newtheorem{remark}{Remark}[section]
\DeclareSymbolFont{bbold}{U}{bbold}{m}{n}
\DeclareSymbolFontAlphabet{\mathbbold}{bbold}

\newcommand{\Xset}{\left\{X_1,\ldots, X_n\right\}}
\newcommand{\reals}{\mathbb{R}}					          

\newcommand{\Ex}{{\rm E}}				                  
\newcommand{\sphere}{\mathbb{S}^d}				         
				                  
\newcommand{\Y}{Y_{\ell, m}}    
\newcommand{\betac}{\beta_{j,k}}                              
\newcommand{\barg}{\frac{\ell}{B^j}}                          
						         
\newcommand{\Yset}{\left\{Y_1,\ldots, Y_n\right\}}
	 				     
\newcommand{\cubew}{\lambda_{j,k}}
\newcommand{\cubep}{\xi_{j,k}}

\newcommand{\ind}[1]{\mathbbold{1}{\bbra{#1}}}
\newcommand{\suml}[1]{\sum_{\ell \in \Lambda_j}{#1}}
\newcommand{\Lambdaj}{\Lambda_j}		
\newcommand{\fest}{{\hat f}_n}
\newcommand{\Ltwo}{{L^2\bra{\mathbb{S}^d}}}
\newcommand{\Lp}{{L^p\bra{\mathbb{S}^d}}}
\newcommand{\besov}{\mathcal{B}_{p,q}^s}
\newcommand{\betaest}{\widehat{\beta}_{j,k}}
\newcommand{\tauthres}{\tau_j}
\newcommand{\besovgen}[3]{\mathcal{B}_{{#1},{#2}}^{#3}}
\newcommand{\Thetaest}{\widehat{\Theta}_j(p)}

\newcommand{\noise}{\varepsilon_i}
\newcommand{\noiseset}{\left\{\varepsilon_1,\ldots, \varepsilon_n\right\}}

\newcommand{\M}{\mathcal{M}}

\newcommand{\eigen}{\gamma}

\providecommand{\abs}[1]{\left\vert#1\right\vert}		
\providecommand{\norm}[1]{\left\Vert#1\right\Vert}			
\providecommand{\bra}[1]{\left(#1\right)}			          
\providecommand{\sbra}[1]{\left[#1\right]}			        
\providecommand{\bbra}[1]{\left\{#1\right\}}			      
\providecommand{\needlet}[1]{\psi_{j,k}\left(#1\right)}		
\providecommand{\bfun}[1]{b\left(#1\right)}

\renewcommand{\theta}{\vartheta}
\renewcommand{\phi}{\varphi}								
\renewcommand{\ll}{\ell}							         
\renewcommand{\l}{\left}                                      
\renewcommand{\r}{\right}                                    

\begin{document}
	
	\begin{frontmatter}	
		\title{Adaptive global thresholding on the sphere}
		\author[mainaddress]{Claudio Durastanti\fnref{myfootnote}}
		\ead{claudio.durastanti@gmail.com}
		\address[mainaddress]{Fakult\"at f\"ur Matematik, Ruhr Universit\"at, Bochum}
		\fntext[myfootnote]{The author is supported by Deutsche Forschungsgemeinschaft (DFG) - GRK 2131, ``High-dimensional Phenomena in Probability --- Fluctuations and Discontinuity".}

		\begin{abstract}
\noindent 
This work is concerned with the study of the adaptivity properties of nonparametric regression estimators  over the $d$-dimensional sphere within the global thresholding framework. The estimators are constructed by means of a form of spherical wavelets, the so-called needlets, which enjoy strong concentration properties in both harmonic and real domains. The author establishes the convergence rates of the $L^p$-risks of these estimators, focussing on their minimax properties and proving their optimality over a scale of nonparametric regularity function spaces, namely, the Besov spaces.
	\end{abstract}
		
		
	\end{frontmatter}

\section{Introduction}\label{sec:intro}

The purpose of this paper is to establish adaptivity for the $L^p$-risk of regression function estimators in the nonparametric setting over the $d$-dimensional sphere $\sphere$. The optimality of the $L^p$ risk is established by means of global thresholding techniques and spherical wavelets known as needlets.

Let $(X_1, Y_1), \ldots, (X_n,Y_n)$ be independent pairs of random variables such that, for each $i \in \{ 1, \ldots ,n\}$, $X_i \in \sphere$ and $Y_i \in \mathbb{R}$. The random variables $X_1, \ldots , X_n$ are assumed to be mutually independent and uniformly distributed locations on the sphere. It is further assumed that, for each $i \in \{1, \ldots , n\}$,
\begin{equation}\label{eqn:regression}
Y_i= f\bra{X_i} + \noise ,
\end{equation}
where $f:\sphere \mapsto \reals$ is an unknown bounded function, i.e., there exists $M>0$ such that 
\begin{equation}\label{eqn:boundf}
\sup_{x\in\sphere}\abs{f\bra{x}} \leq M <\infty.
\end{equation}
Moreover, the random variables $\epsilon_1, \ldots , \epsilon_n$ in Eq.~\eqref{eqn:regression} are assumed to be mutually independent and identically distributed with zero mean. Roughly speaking, they can be viewed as the observational errors and in what follows, they will be assumed to be sub-Gaussian. 

In this paper, we study the properties of nonlinear global hard thresholding estimators, in order to establish the optimal rates of convergence of $L^p$-risks for functions belonging to the so-called Besov spaces. 

\subsection{An overview of the literature}\label{sec:overview}

In recent years, the issue of minimax estimation in nonparametric settings has received considerable attention in the statistical inference literature. The seminal contribution in this area is due to Donoho et al. \cite{donoho1}. In this paper, the authors provide nonlinear wavelet estimators for density functions on $\reals$, lying over a wide nonparametric regularity function class, which attain optimal rates of convergence up to a logarithmic factor. Following this work, the interaction between wavelet systems and nonparametric function estimation has led to a considerable amount of developments, mainly in the standard Euclidean framework; see, e.g., \cite{cai1, cai, tsyb, Kerkypicard92,  Kerkypicard93, thres,kpt96} and the textbooks \cite{WASA,tsybakov} for further details and discussions. 

More recently, thresholding methods have been applied to broader settings. In particular, nonparametric estimation results have been achieved on $\sphere$ by using a second generation wavelet system, namely, the spherical needlets. Needlets were introduced by Narcowich et al. \cite{npw2,npw1}, while their stochastic properties dealing with various applications to spherical random fields were examined in \cite{bkmpAoS, cammar, ejslan, spalan,m}. Needlet-like constructions were also established over more general manifolds by Geller and Mayeli \cite{gm3,gm1,gm2,gelpes}, Kerkyacharian et al. \cite{knp} and Pesenson \cite{pesenson} among others, and over spin fiber bundles by Geller and Marinucci \cite{gelmar, gelmar2010}. 

In the nonparametric setting, needlets have found various applications on directional statistics. Baldi et al. \cite{bkmpAoSb} established minimax rates of convergence for the $L^p$-risk of nonlinear needlet density estimators within the hard local thresholding paradigm, while analogous results concerning regression function estimation were established by Monnier \cite{monnier}. The block thresholding framework was investigated in Durastanti \cite{durastanti2}. Furthermore, the adaptivity of nonparametric regression estimators of spin function was studied in Durastanti et al. \cite{dmg}. In this case, the regression function takes as its values algebraical curves lying on the tangent plane for each point on $\mathbb{S}^2$ and the wavelets used are the so-called spin (pure and mixed) needlets; see Geller and Marinucci \cite{gelmar, gelmar2010}. 

The asymptotic properties of other estimators for spherical data, not concerning the needlet framework, were investigated by Kim and Koo \cite{kimkoo,kim,kookim}, while needlet-like nearly-tight frames were used in Durastanti \cite{durastanti3} to establish the asymptotic properties of density function estimators on the circle. Finally, in Gautier and Le Pennec \cite{gautier}, the adaptive estimation by needlet thresholding was introduced in the nonparametric random coefficients binary choice model. Regarding the applications of these methods in practical scenarios, see, e.g., \cite{fayetal13,fay08,  iuppa}, where they were fruitfully applied to some astrophysical problems, concerning, for instance, high-energy cosmic rays and Gamma rays.     

\subsection{Main results}\label{sub:main}

Consider the regression model given in Eq.~\eqref{eqn:regression} and let $\{\psi_{j,k}: j \geq 0, k=1,\ldots,K_j\}$ be the set of $d$-dimensional spherical needlets. Roughly speaking, $j$ and $K_j$ denote the resolution level $j$ and the cardinality of needlets at the resolution level $j$, respectively. The regression function $f$ can be rewritten in terms of its needlet expansion. Namely, for all $x \in \sphere$, one has
\begin{equation*}\label{eqn:needexp}
f\bra{x} =\sum_{j \geq 0} \sum_{k=1}^{K_j} \betac \needlet{x},
\end{equation*}
where $\{\betac: j \geq 0, k=1,\ldots,K_j\}$ is the set of needlet coefficients.

For each $j\geq 0$ and $k \in \{ 1,\ldots, K_j\}$, a natural unbiased estimator for $\betac$ is given by the corresponding empirical needlet coefficient, viz.
\begin{equation}\label{eqn:needest}
\betaest=\frac{1}{n}\sum_{i=1}^n Y_i \needlet{X_i};
\end{equation}
see, e.g., Baldi et al. \cite{bkmpAoSb} and H\"ardle et al. \cite{WASA}. Therefore, the global thresholding needlet estimator of $f$ is given, for each $x \in \sphere$, by
\begin{equation}\label{eqn:fest}
{\hat f}_n\bra{x} =\sum_{j = 0}^{J_n}\tauthres \sum_{k=1}^{K_{J_n}} \betaest \needlet{x},
\end{equation}
where $\tauthres$ is a nonlinear threshold function comparing the given $j$-dependent statistic $\Thetaest$, built on a subsample of $p<n$ observations, to a threshold based on the observational sample size. If $\Thetaest$ is above the threshold, the whole $j$-level is kept; otherwise it is discarded. 

Loosely speaking, this procedure allows one to delete the coefficients corresponding to a resolution level $j$ whose contribution to the reconstruction of the regression function $f$ is not clearly distinguishable from the noise. Following Kerkyacharian et al. \cite{kpt96}, we consider the so-called hard thresholding framework, defined as
\begin{equation*}\label{eqn:thres}
\tauthres =\tauthres (p)= \mathbbold{1} \{ \hat{\Theta}_j(p) \geq B^{dj}n^{-p/2}\},
\end{equation*}   
where $p \in \mathbb{N}$ is even. Further details regarding the statistic $\hat{\Theta}_j (p)$ will be discussed in Section \ref{sub:thres}, where the choice of the threshold $B^{dj}n^{-p/2}$ will also be motivated. 

For the rest of this section, we consider $\Thetaest$ as an unbiased statistic of $|\beta_{j,1}|^p + \cdots + |\beta_{j,K_j}|^p$. The so-called truncation bandwidth $J_n$, on the other hand, is the higher frequency on which the empirical coefficients ${\hat \beta}_{j,1} , \ldots , {\hat \beta}_{j,K_j}$ are computed. The optimal choice of the truncation level is $J_n = \ln _B (n^{1/d})$; for details, see Section \ref{sec:global}. This allows the error due to the approximation of $f$, which is an infinite sum with respect to $j$, to be controlled by a finite sum, such as the estimator $\fest$.
 
Our objective is to estimate the global error measure for the regression estimator $\fest$. For this reason, we study the worst possible performance over a so-called nonparametric regularity class $\{\mathcal{F}_\alpha: \alpha \in A\}$ of function spaces  of the $L^p$-risk, i.e.,
\begin{equation*}
R_n \bigl( {\hat f}_n;\mathcal{F_\alpha} \bigr) =\sup_{f \in \mathcal{F_\alpha}}\Ex\Big(\| {\hat f}_n - f\|_\Lp ^p\Bigr).
\end{equation*}
Recall that an estimator ${\hat f}_n$ is said to be adaptive for the $L^p$-risk and for the scale of classes $\{\mathcal{F}_\alpha: \alpha \in A\}$ if, for every $\alpha \in A$, there exists a constant $c_\alpha>0$ such that 
\begin{equation*}
\Ex\Big( \| {\hat f}_n - f\|_\Lp^p \Bigr) \leq c_\alpha R_n \big({\hat f}_n;\mathcal{F}_\alpha\bigr);
\end{equation*} 
see, e.g., \cite{bkmpAoSb,WASA,kpt96}.

For $r>0$ and for $p \in [1,r]$, we will establish that the regression estimator $\fest$ is adaptive for the class of Besov spaces $\besovgen{p}{q}{s}$, where $1\leq q \leq \infty$ and $d/p\leq s < r+1$. Finally, let $R \in (0,\infty)$ be the radius of the Besov ball on which $f$ is defined. The proper choice of $r$ will be motivated in Section \ref{sub:harm}. Our main result is described by the following theorem.

\begin{theorem}\label{thm:main}
Given $r\in\bra{1,\infty}$, let $p \in [1,r]$. Also, let ${\hat f}_n$ be given by Eq.~\eqref{eqn:fest}, with $J_n=\ln_B n^{1/d}$. Then, for $1\leq q \leq \infty$,  $d/p\leq s < r+1$ and $0<R<\infty$, there exists $C>0$ such that 
\begin{equation*}
\sup_{f \in \besovgen{r}{q}{s}\bra{R}} \Ex\Bigl( \| {\hat f}_n - f\|_\Lp^p \Bigr)\leq C n^{\frac{-sp}{2s+d}} .
\end{equation*}
\end{theorem}	

The behavior of the $L^\infty$-risk function will be studied separately in Section~\ref{sec:global} and the analogous result is described in Theorem \ref{thm:infty}. Moreover, the details concerning the choice of $r$ will be presented in Remark \ref{rem:pho2} and other properties of $L^p$-risk functions, such as optimality, will be discussed in Remark \ref{rem:remark}.

\subsection{Comparison with other results}\label{sub:comparison}

The bound given in Eq.~\eqref{eqn:main} is consistent with the results of Kerkyacharian et al. \cite{kpt96}, where global thresholding techniques were introduced on $\reals$. As far as nonparametric inference over spherical datasets is concerned, our results can be viewed as an alternative proposal to the existing nonparametric regression methods (see, e.g., \cite{bkmpAoSb, durastanti2, dmg, monnier}), related to the local and block thresholding procedures. 

Recall that in local thresholding paradigm, each empirical estimator $\betaest$ is compared to a threshold $\tau_{j,k}$ and it is, therefore, kept or discarded if its absolute value is above or below $\tau_{j,k}$ respectively, i.e., the threshold function is given by $\mathbbold{1} \{ |\betaest|\geq \tau_{j,k}\}$. Typically, the threshold is chosen such that $\tau_{j,k} = \kappa \bra{\ln n/n}$, where $\kappa$ depends explicitly on two parameters, namely, the radius $R$ of the Besov ball on which the function $f$ is defined and its supremum $M$; see, e.g., Baldi et al. \cite{bkmpAoSb}. An alternative and partially data-driven choice for $\kappa$ is proposed by Monnier \cite{monnier}, i.e., here 
$$
\kappa= \frac{\kappa_0}{n} \sum_{i=1}^n \needlet{X_i}^2. 
$$
Even if this stochastic approach is proved to outperform the deterministic one, the threshold still depends on both $R$ and $M$, which control $\kappa_0$. Also according to the results established on $\reals$ (see H\"ardle et al. \cite{WASA}), local techniques entail nearly optimality rates for the $L^p$-risks over a wide variety of regularity function spaces. In this case, the regression function $f$ belongs to $\besov\bra{R}$, where $s \geq d/r$, $p \in \bbra{1,\infty}$, $q \in \bbra{1,\infty}$ and $0<R<\infty$ (cf. \cite{bkmpAoSb,dmg,WASA}). However, these adaptive rates of convergence are achieved on the expense of having an extra logarithmic term and of requiring explicit knowledge of the radius of the Besov balls on which $f$ is defined, in order to establish an optimal threshold.

As far as the block thresholding is concerned, for any fixed resolution level this procedure collects the coefficients ${\hat \beta}_{j,1}, \ldots, {\hat \beta}_{j,K_j}$ into $\ell= \ell \bra{n}$ blocks denoted $B_{j,1}, \ldots, B_{j,\ell}$ of dimension depending on the sample size. Each block is then compared to a threshold and then it is retained or discarded. This method has exact convergence rate (i.e., without the logarithmic extra term), although it requires explicit knowledge of the Besov radius $R$. Furthermore, the estimator is adaptive only over a narrower subset of the scale of Besov spaces, the so-called regular zone; see H\"ardle et al. \cite{WASA}. The construction of blocks on $\sphere$ can also be a difficult procedure, as it requires a precise knowledge of the pixelization of the sphere, namely, the structure of the subregions on which the sphere is partitioned, in order to build spherical wavelets.

On the other hand, the global techniques presented in this paper do not require any knowledge regarding the radius of Besov ball and have exact optimal convergence rates even over the narrowest scale of regularity function spaces. 

\subsection{Plan of the paper}\label{sub:plan}

This paper is organized as follows. Section~\ref{sec:prel} presents some preliminary results, such as the construction of spherical needlet frames on the sphere, Besov spaces and their properties. In Section~\ref{sec:global}, we describe the statistical methods we apply within the global thresholding paradigm. This section also includes an introduction to the properties of the sub-Gaussian random variables and of the $U$-statistic $\Thetaest$, which are key for establishing the thresholding procedure. Section \ref{sec:numerical} provides some numerical evidence. Finally, the proofs of all of our results are collected in Section~\ref{sec:proofs}.

\section{Preliminaries}\label{sec:prel}

This section presents details concerning the construction of needlet frames, the definition of spherical Besov spaces and their properties. In what is to follow the main bibliographical references are \cite{bkmpAoSb, bkmpAoS, donoho1, gelpes,WASA, tsyb,MaPeCUP, npw2, npw1}.  

\subsection{Harmonic analysis on $\sphere$ and spherical needlets}\label{sub:harm}

Consider the simplified notation $\Ltwo=L^2\bra{\sphere,dx}$, where $dx$ is the uniform Lebesgue measure over $\sphere$. Also, let $\mathcal{H}_\ell$ be the restriction to $\sphere$ of the harmonic homogeneous polynomials of degree $\ell$; see, e.g., Stein and Weiss \cite{steinweiss}. Thus, the following decomposition holds
\begin{equation*}
\Ltwo = \bigoplus_{\ell=0}^\infty \mathcal{H}_\ell.
\end{equation*}

An orthonormal basis for $\mathcal{H}_\ell$ is provided by the set of spherical harmonics $\{ \Y : m=1,\ldots,g_{\ell,d} \}$ of dimension $g_{\ell,d}$ given by
\begin{equation*}
g_{\ell,d}=\frac{\ell+\eta_d}{\eta_d} \binom{\ell+2\eta_d-1}{\ll}, \quad \eta_d =\frac{d-1}{2}. 
\end{equation*}
For any function $f \in \Ltwo$, we define the Fourier coefficients as
\begin{equation*}
a_{\ell,m} :=\int_{\sphere}\overline{\Y}\bra{x} f\bra{x} dx,
\end{equation*}
such that the kernel operator denoting the orthogonal projection over $\mathcal{H}_\ell$ is given, for all $ x \in \sphere$, by
\begin{equation*} 
P_{\ell,d}f \bra{x}=\sum_{m=1}^{g_{\ell,d}} a_{\ell,m}\Y \bra{x}.
\end{equation*} 
Also, let the measure of the surface of $\sphere$ be given by 
\begin{equation*}
\omega_d= 2\pi^{(d+1)/2}{\Big /}{\Gamma\bra{\frac{d+1}{2}}}.
\end{equation*}

The kernel associated to the projector $P_{\ell,d}$ links spherical harmonics to the Gegenbauer polynomial of parameter $\eta_d$ and order $\ell$, labelled by $\mathcal{C}_\ell^{\bra{\eta_q}}$. Indeed, the following summation formula holds 
\begin{equation*}
P_{\ell,d}\bra{x_1,x_2}=\sum_{m=1}^{g_{\ell,d}}\overline{\Y}\bra{x_1}\Y{x_2}=\frac{\ell+\eta_d}{\eta_d \omega_d} \mathcal{C}_\ell^{\bra{\eta_d}}\bra{\langle x_1,x_2\rangle},
\end{equation*}
where $\langle \cdot,\cdot\rangle$ is the standard scalar product on $\reals^{d+1}$; see, e.g., Marinucci and Peccati \cite{MaPeCUP}.

Following Narcowich et al. \cite{npw1}, $\mathcal{K}_\ell = \oplus_{i=0}^{\ell} \mathcal{H} _i$ is the linear space of homogeneous polynomials on $\sphere$ of degree smaller or equal to $\ll$; see also \cite{bkmpAoSb, MaPeCUP,npw2}. Thus, there exist a set of positive cubature points $\mathcal{Q}_\ell\in \sphere$ and a set of cubature weights $\bbra{\lambda_\xi}$,  indexed by $\xi\in \mathcal{Q}_\ll$, such that, for any $f \in \mathcal{K}_\ell$,
\begin{equation*}
\int_{\sphere} f\bra{x}dx= \sum_{\xi \in \mathcal{Q}_\ll }\lambda_\xi f\bra{\xi}.
\end{equation*}

In the following, the notation $a \approx b$ denotes that there exist $c_1, c_2 >0$ such that $c_1 b \leq a \leq c_2 b$. For a fixed resolution level $j$ and a scale parameter $B$, let $K_j=\text{card}\bra{\mathcal{Q}_{\sbra{2B^{j+1}}}}$. Therefore, $\{\xi_{j,k} : k=1,\ldots, K_j\}$ is the set of cubature points associated to the resolution level $j$, while $\{\lambda_{j,k}: k=1,\ldots,K_j \}$ contains the corresponding cubature weights. These are typically chosen such that
$$
K_j \approx B^{d j} \quad \mbox{and} \quad \forall_{k \in \{ 1,\ldots,K_j\} } ~~ \lambda_{j,k} \approx B^{-d j}. 
$$
Define the real-valued weight (or window) function $b$ on $\bra{0,\infty}$ so that
\begin{itemize}
	\item [(i)] $b$ lies on a compact support $\sbra{B^{-1},B}$;
	\item [(ii)] the partitions of unity property holds, namely, $\sum_{j\geq0} b^2(\ell/B^j)=1$, for $\ell \geq B$;
	\item [(iii)] $b \in C^{\rho}\bra{0,\infty}$ for some $\rho\geq 1$.
\end{itemize}

\begin{remark}\label{rem:pho}

Note that $\rho$ can be either a positive integer or equal to $\infty$. In the first case, the function $b(\cdot) $ can be built by means of a standard B-spline approach, using linear combinations of the so-called Bernstein polynomials, while in the other case, it is constructed by means of integration of scaled exponential functions (see also Section \ref{sec:numerical}). Further details can be found in the textbook Marinucci and Peccati \cite{MaPeCUP}.  
\end{remark}

For any $j\geq 0$ and $k \in \{ 1,\ldots,K_j\}$, spherical needlets are defined as
\begin{equation*}
\needlet{x} = \sqrt{\cubew} \, \sum_{\ell \geq 0}{\bfun{\barg}P_{\ell,d}\bra{x,\cubep}}.
\end{equation*}    
Spherical needlets feature some important properties descending on the structure of the window function $b$. Using the compactness of the frequency domain, it follows that $\psi_{j,k}$ is different from zero only on a finite set of frequencies $\ell$, so that we can rewrite the spherical needlets as  
\begin{equation*}\label{eqn:needdef}
\needlet{x} = \sqrt{\cubew} \suml{\bfun{\barg}P_{\ell,d}\bra{x,\cubep}},
\end{equation*}    
where $\Lambdaj=\bbra{u:u \in \left(\sbra{B^{j-1}},\sbra{B^{j+1}}\right)}$ and $\sbra{u}$, $u \in \reals$, denotes the integer part of $u$. From the partitions of unity property, the spherical needlets form a tight frame over $\sphere$ with unitary tightness constant. For $f\in\Ltwo$, 
\begin{equation*}
\norm{f}_{\Ltwo}^2=\sum_{j\geq 0}\sum_{k=1}^{K_j}\abs{\betac}^2,
\end{equation*}
where
\begin{equation}\label{eqn:betacoeff}
\betac=\int_{\sphere} f\bra{x} \needlet{x}dx,
\end{equation}
are the so-called needlet coefficients.
Therefore, we can define the following reconstruction formula (holding in the $L^2$-sense): for all $x \in \sphere$,
\begin{equation*}\label{eqn:rec}
f\bra{x}=\sum_{j\geq 0}\sum_{k=1}^{K_j} \betac \needlet{x}.
\end{equation*}
From the differentiability of $b$, we obtain the following quasi-exponential localization property; for $x \in \sphere$ and any $\eta \in \mathbb{N}$ such that $\eta \leq \rho$, there exists $c_\eta>0$ such that 
\begin{equation}\label{eqn:needloc}
\abs{\needlet{x}}\leq \frac{c_\eta B^{j {d}/{2}}}{\{1+B^{j {d}/{2}}d\bra{x,\xi_{j,k}}\}^\eta},
\end{equation}
where $d\bra{\cdot,\cdot}$ denotes the geodesic distance over $\sphere$.

Roughly speaking, $\abs{\needlet{x}} \approx B^{j{d}/{2}}$ if $x$ belongs to the pixel of area $B^{-dj}$ surrounding the cubature point $\xi_{j,k}$; otherwise, it is almost negligible. The localization result yields a similar boundedness property for the $L^p$-norm, which is crucial for our purposes. In particular, for any $p \in \left[\left.1,\infty\right)\right.$ there exist two constants $c_p, C_p >0$ such that 
\begin{equation}\label{eqn:Lpnorm}
c_p B^{jd\bra{\frac{1}{2}-\frac{1}{p}} }\leq \norm{\psi_{j,k}}_{\Lp} \leq C_p B^{jd\bra{\frac{1}{2}-\frac{1}{p}}},
\end{equation}
and there exist two constants $c_\infty, C_\infty >0$ such that
\begin{equation*}
c_\infty B^{j\frac{d}{2} }\leq \norm{\psi_{j,k}}_{L^\infty\bra{\sphere}} \leq C_\infty B^{j {d}/{2} }.
\end{equation*}

According to Lemma 2 in Baldi et al. \cite{bkmpAoSb}, the following two inequalities hold. For every $0<p\leq\infty$,
\begin{equation}\label{eqn:boundsuppa}
\norm{\sum_{k=1}^{K_j}\betac \psi_{j,k}}_{\Lp} \leq cB^{jd\bra{\frac{1}{2}-\frac{1}{p}}}\norm{\beta_{j,k}}_{\ll_p},
\end{equation}
and for every $1\leq p\leq \infty$,
\begin{equation*}\label{eqn:boundsuppa2}
\norm{\beta_{j,k}}_{\ll_p} B^{jd\bra{\frac{1}{2}-\frac{1}{p}}}\leq c \norm{f}_{\Lp},
\end{equation*}
where $\ll_p$ denotes the space of $p$-summable sequences. The generalization for the case $p=\infty$ is trivial.

The following lemma presents a result based on the localization property.

\begin{lemma}\label{lemma:needloc}
For $x \in \sphere$, let $\needlet{x}$ be given by Eq.~\eqref{eqn:needdef}. Then, for $q \geq 2$,  $k_{i_1} \neq k_{i_2}$, for $i_1 \neq i_2=1,\ldots,q$, and for any $\eta \geq 2$, there exists $C_\eta>0$ such that
\begin{equation*}\label{eqn:needloclemma}
\int_{\sphere} \prod_{i=1}^q \psi_{j,k_i}\bra{x}dx \leq \frac{B^{dj\bra{q-1}}}{\bra{1+B^{dj}\Delta}^{\eta\bra{q-1}}},
\end{equation*}
where 
$$
\Delta=\min _{i_1,i_2 \in \{ 1,\ldots,q\}, i_1\neq i_2} d (\xi_{j,k_{i_1}},\xi_{j,k_{i_2}}).
$$
\end{lemma}
\begin{remark}

As discussed in Geller and Pesenson \cite{gelpes} and Kerkyacharian et al. \cite{knp}, needlet-like wavelets can be built over more general spaces, namely, over compact manifolds. In particular, let $\left\{\M,g\right\}$ be a smooth compact homogeneous manifold of dimension $d$, with no boundaries. For the sake of simplicity, we assume that there exists a Laplace--Beltrami operator on $\M$ with respect to the action $g$, labelled by $\Delta_{\M}$. The set $\{\gamma_q: q \geq 0\}$ contains the eigenvalues of $\Delta_{\M}$ associated to the eigenfunctions $\{u_q : q\geq 0\}$, which are orthonormal with respect to the Lebesgue measure over $\M$ and they form an orthonormal basis in $L^2\bra{\M}$; see \cite{gm2,gelpes}. Every function $f \in L^2\bra{\M}$ can be described in terms of its harmonic coefficients, given by $a_q = \langle f, u_q\rangle_{\L^2\bra{\M}}$, so that, for all $x \in \M$,
	\begin{equation*}
	f\bra{x}=\sum_{q\geq 1}a_q u_q \bra{x}.
	\end{equation*}
Therefore, it is possible to define a wavelet system over $\bbra{\M,g}$ describing a tight frame over $\M$ along the same lines as in Narcowich et al. \cite{npw1} for $\sphere$; see also \cite{gelpes,knp,pesenson} and the references therein, such as Geller and Mayeli \cite{gm1,gm2}. Here we just provide the definition of the needlet (scaling) function on $\M$, given by 
$$
	\needlet{x} =\sqrt{\cubew}\sum_{q=B^{j-1}}^{B^{j+1}} \bfun{\frac{\sqrt{-\eigen_q}}{B^{j}}}u_q\bra{x}\bar{u}\bra{\xi_{j,k}},
$$
where in this case the set $\bbra{\cubep,\cubew}$ characterizes a suitable partition of $\M$, given by a $\varepsilon$-lattice on $\M$, with $\varepsilon=\sqrt{\cubew}$. Further details and technicalities concerning $\varepsilon$-lattices can be found in Pesenson \cite{pesenson}. Analogously to the spherical case, for $f \in L^2\bra{\M}$ and arbitrary $j\geq 0$ and $k \in \{ 1,\ldots, K_j\}$, the needlet coefficient corresponding to $\psi_{j,k}$ is given by
$$
	\betac=\langle f, \psi_{j,k} \rangle_{\Ltwo} =\sqrt{\cubew}\sum_{q=B^{j-1}}^{B^{j+1}} \bfun{\frac{\sqrt{-\eigen_q}}{B^{j}}}a_q u_{q}\bra{\xi_{j,k}}.
$$
These wavelets preserve all the properties featured by needlets on the sphere: because, as shown in the following sections, the main results presented here do not depend strictly on the underlying manifold (namely, the sphere) but rather they can be easily extended to more general frameworks such as compact manifolds, where the concentration properties of the wavelets and the smooth approximation properties of Besov spaces still hold.
\end{remark}

\subsection{Besov space on the sphere}\label{sub:besov}

Here we will recall the definition of spherical Besov spaces and their main approximation properties for wavelet coefficients. We refer to \cite{bkmpAoSb,dmg,WASA,npw2} for more details and further technicalities.

Suppose that one has a scale of functional classes $\mathcal{G}_t$, depending on the $q$-dimensional set of parameters $t \in T \subseteq \reals^q$. The approximation error $G_t\bra{f;p}$ concerning the replacement of $f$ by an element $g \in \mathcal{G}_t$ is given by
\begin{equation*}
G_t\bra{f;p}=\inf_{g \in \mathcal{G}_t}\norm{f-g}_\Lp.
\end{equation*}
Therefore, the Besov space $\besov$ is the space of functions such that $f \in \Lp$ and 
\begin{equation*}
\sum_{t\geq 0}\frac{1}{t} \{ t^s G_t\bra{f;p}\}^q <\infty,
\end{equation*}
which is equivalent to
\begin{equation*}
\sum_{j\geq 0}B^j\{G_{B^j}\bra{f;p}\}^q <\infty.
\end{equation*}
The function $f$ belongs to the Besov space $\besov$ if and only if 
\begin{equation}\label{eqn:bbes}
\left[ \sum_{k=1}^{K_j}\{\abs{\betac}\norm{\psi_{j,k}}_{\Lp}\}^p\right]^{{1}/{p}} = B^{-js} w_j,
\end{equation}
where $w_j\in \ell_q$, the standard space of $q$-power summable infinite sequences.
Loosely speaking, the parameters $s\geq 0$, $1\leq p \leq \infty$ and $1\leq q \leq \infty$ of the Besov space $\besov$ can be viewed as follows: given $B>1$, the parameter $p$ denotes the $p$-norm of the wavelet coefficients taken at a fixed resolution $j$, the parameter $q$ describes the weighted $q$-norm taken across the scale $j$, and the parameter $r$ controls the smoothness of the rate of decay across the scale $j$. In view of Eq.~\eqref{eqn:Lpnorm}, the Besov norm is defined as 
\begin{align*}
\norm{f}_{\besov} = & \norm{f}_{\Lp} +\left\{\sum_{j\geq 0} B^{jq\{s+d\bra{{1}/{2}-{1}/{p}}\}}\bra{\sum_{k=1}^{K_j}\abs{\betac}^p}^{{q}/{p}}\right\}^{ {1}/{q}}\\
=& \norm{f}_{\Lp} +\norm{ B^{j\{s+d\bra{{1}/{2}- {1}/{p}}\}} \norm{\betac}_{\ell_p}}_{\ell_q},
\end{align*}
for $q\geq1$. The extension to the case $q=\infty$ is trivial. 

We conclude this section by introducing the Besov embedding, discussed in \cite{bkmpAoSb, kerkypicard,kpt96} among others. For $p<r$, one has 
$$
\besovgen{r}{q}{s} \subset \besovgen{p}{q}{s} \quad \mbox{and} \quad \besovgen{p}{q}{s} \subset \besovgen{r}{q}{s-d\bra{ {1}/{p}- {1}/{r}}}, 
$$
or, equivalently,
\begin{align}
\sum_{k=1}^{K_j} \abs{\betac}^p \leq &\bra{\sum_{k=1}^{K_j} \abs{\betac}^r} K_j^{1-{p}/{r}};\label{eqn:besovmin}\\
\sum_{k=1}^{K_j} \abs{\betac}^r \leq & \sum_{k=1}^{K_j} \abs{\betac}^p.\label{eqn:besovmax}
 \end{align}
Proofs and further details can be found, for instance, in \cite{bkmpAoSb, dmg}. 

\section{Global thresholding with spherical needlets}\label{sec:global}

This section provides a detailed description of the global thresholding technique applied to the nonparametric regression problem on the $d$-dimensional sphere. We refer to \cite{efrom, WASA,kpt96} for an extensive description of global thresholding methods and to \cite{bkmpAoSb,dmg} for further details on nonparametric estimation in the spherical framework.   

\subsection{The regression model}\label{sub:regression}

Recall the regression model given by Eq.~\eqref{eqn:regression}, i.e., for all $i \in \{1, \ldots, n\}$,
\begin{equation*}
Y_i= f\bra{X_i} + \noise.
\end{equation*}
While $\Xset$ denotes the set of uniformly sampled random directions over $\sphere$, $\Yset$ is the set of the independent observations which are related to $\Xset$ through the regression function $f$ and affected by $\noiseset$, which is the set of the observational errors. The independent and identically distributed random variables $\varepsilon_1, \ldots , \varepsilon_n$ are such that, for all $i \in \{ 1, \ldots , n\}$,
\begin{equation*}
\Ex\bra{\noise}=0, \quad \Ex\bra{\noise^2}=\sigma_\varepsilon^2 < \infty, \label{eqn:varnoise}
\end{equation*}
and they are assumed to be sub-Gaussian. Further details are given in Section~\ref{sub:noise}.
Assume that $f \in \besov$, $d/p\leq s < r+1$, $1\leq p \leq r$ and $1\leq q\leq \infty$, where $r$ is fixed, and that there exists $R>0$ such that $\norm{f}_{\besov} \leq R$.  
As mentioned in Section \ref{sub:main} and Section \ref{sec:prel}, the regression function can be expanded in terms of needlet coefficients as 
\begin{equation*}
f\bra{x}=\sum_{j \geq 0} \sum_{k=1}^{K_j} \betac \needlet{x},
\end{equation*}
where $\betac$ are given in Eq. \eqref{eqn:betacoeff}. 

\begin{remark}\label{rem:pho2}
As discussed in Section \ref{sub:comparison}, we do not require explicit knowledge of the Besov radius $R$. Although it can be difficult to determine $r$ explicitly, we suggest the following criterion. Consider Remark \ref{rem:pho}; if $\rho < \infty$, we choose $r=\rho$ (see again \cite{kpt96}). If $\rho=\infty$, we choose $r=B^{d\bra{J_n+1}}$ empirically, using the so-called vanishing moment condition on $\sphere$, properly adapted for the needlet framework; see, e.g., Schr\"oder and Sweldens \cite{schro}.
\end{remark}

\subsection{The observational noise}\label{sub:noise}

Following Durastanti et al. \cite{dmg}, we assume that $\varepsilon_1, \ldots, \varepsilon_n$ follow a sub-Gaussian distribution; see also Buldygin and Kozachenko \cite{buko}). 
A random variable $\varepsilon$ is said to be sub-Gaussian of parameter $a$ if, for all $\lambda \in \reals$, there exists $a\geq 0$ such that
\begin{equation*}
\Ex (e^{\lambda \varepsilon}) \leq e^{{a^2\lambda^2}/{2}} .
\end{equation*} 
Sub-Gaussian random variables are characterized by the sub-Gaussian standard, given by
\begin{equation*}
\zeta\bra{\varepsilon} := \inf\bbra{a \geq 0: \Ex (e^{\lambda \varepsilon}) \leq e^{{a^2\lambda^2}/{2}} , \lambda \in \reals},
\end{equation*}
which is finite. As proved in \cite{buko},
$$
\zeta\bra{\varepsilon} =  \sup_{\lambda \neq 0} \bbra{\frac{2\ln \Ex \bra{ e^{\lambda\varepsilon}}}{\lambda^2}}^{{1}/{2}};\quad \Ex (e^{\lambda \varepsilon}) \leq  \exp \left\{\frac{\lambda^2 \zeta^2\bra{\varepsilon}}{2}\right\}.
$$
Following Lemma 1.4 in \cite{buko}, for $p>0$
\begin{align*}
\Ex\bra{\varepsilon}=0;\ \Ex\bra{\varepsilon^2}\leq \zeta\bra{\varepsilon}; \ \Ex\bra{\abs{\varepsilon}^p} \leq 2 \bra{\frac{p}{\exp}}^{{p/2}} \zeta^p\bra{\varepsilon}.
\end{align*}

Therefore, sub-Gaussian random variables are characterized by the same moment inequalities and concentration properties featured by null-mean Gaussian or bounded random variables. 
\begin{remark}In order to establish the probabilistic bounds described in Sections \ref{sub:est} and \ref{sub:thres}, it would be sufficient for $\varepsilon_1, \ldots, \varepsilon_n$ to be null-mean independent random variables with finite absolute $p$th moment. However, we include the notion of sub-Gaussianity in order to be consistent with the existing literature; see \cite{dmg}. Furthermore, sub-Gaussianity involves a wide class of random variables, including Gaussian and bounded random variables and, in general, all the random variables such that their moment generating function has a an upper bound in terms of the moment generating function of a centered Gaussian random variable of variance $a$. Hence the term ``sub-Gaussian." 
\end{remark}

\subsection{The estimation procedure}\label{sub:est}

We note again that the method established here can be viewed as an extension of global thresholding techniques to the needlet regression function estimation. In this sense, our results are strongly related to those presented in \cite{bkmpAoSb, dmg, kpt96}, as discussed in Section \ref{sub:comparison}.

For any $j \geq 0$ and $k \in \{ 1,\ldots, K_j\}$, the empirical needlet estimator is given by
\begin{equation*}
\betaest=\frac{1}{n} \sum_{i=1}^{n} Y_i \needlet{X_i}, 
\end{equation*} 
and it is unbiased, i.e.,
\begin{equation*}
\Ex (\betaest) =\frac{1}{n} \sum_{i=1}^{n} [\Ex \{f\bra{X_i} \needlet{X_i} \} + \Ex (\varepsilon_i) \Ex\{ \needlet{X_i}\}]=\betac. 
\end{equation*} 
The empirical needlet coefficients are moreover characterized by the following stochastic property.
\begin{proposition}\label{prop:betac}
	Let $\betac$ and $\betaest$ be as in Eq.~\eqref{eqn:betacoeff} and Eq.~\eqref{eqn:needest}, respectively. Thus, for $p\geq 1$, there exists $\tilde{c}_p$ such that
	\begin{equation*}\label{eqn:betamoment}
	\Ex ( | {\hat \beta}_{j,k} - \beta_{j,k}|^p) \leq \tilde{c}_p n^{-{p}/{2}}.
	\end{equation*}
\end{proposition}
Therefore, we define the global thresholding needlet regression estimator at every $x \in \sphere$ by
\begin{equation*}
{\hat f}_n\bra{x} =\sum_{j = 0}^{J_n}\tauthres \sum_{k} \betaest \needlet{x};
\end{equation*}
see Eq.~\eqref{eqn:fest}. 
Recall now the main results, stated in Section \ref{sec:intro}. 
\begin{theorem*}
Given $r\in\bra{1,\infty}$, let $p \in [1,r]$. Also, let ${\hat f}_n$ be given by Eq.~\eqref{eqn:fest}, with $J_n=\ln_B n^{1/d}$. Then, for $1\leq q \leq \infty$,  $d/p\leq s < r+1$ and $0<R<\infty$, there exists $C>0$ such that 
\begin{equation}\label{eqn:main}
\sup_{f \in \besovgen{r}{q}{s}\bra{R}} \Ex\Bigl( \| {\hat f}_n - f\|_\Lp^p \Bigr)\leq C n^{\frac{-sp}{2s+d}} .
\end{equation}
\end{theorem*}
\noindent In the nonparametric thresholding settings, the $L^p$-risk is generally bounded as follows 
\begin{eqnarray*}\label{eqn:stochbias} 
\notag \Ex\bra{\| {\hat f}_n - f\|_\Lp^p} &\leq & 2^{p-1}\left\{\Ex\bra{\Bigl\| \sum_{j=0}^{J_n}(\betaest-\betac)\psi_{j,k} \Bigr\|_\Lp^p}
\right.	\\
&& \qquad +\left.\norm{\sum_{j\geq J_n}\betac\psi_{j,k}}_\Lp^p\right\} \notag\\
&=& \mathfrak{S} +\mathfrak{B},
\end{eqnarray*}
where $\mathfrak{S}$ is the stochastic error, due to the randomness of the observations and $\mathfrak{B}$ is the (deterministic) bias error.
The so-called truncation level $J_n$ is chosen so that $B^{dJ_n}=n$. 

In this case the bias term term does not affect the rate of convergence for $s\in \bra{d/p,r+1}$. As far as $\mathfrak{S}$ is concerned, its asymptotic behavior is established by means of the so-called optimal bandwidth selection, i.e., a frequency $J_s$ such that $B^{dJ_s}=n^{\frac{1}{\bra{2s+d}}}$; see \cite{efrom,kpt96}. Note that trivially, $J_s < J_n$. The meaning and the relevance of the optimal bandwidth selection is given in Section~\ref{sec:proofs}, in the proof of Theorem \ref{thm:main}. However, in the next section it will also be crucial for the construction of the threshold function.

Consider now the case $p=\infty$. First, we have to modify the threshold function given in Eq.~\eqref{eqn:threshold} slightly, in view of the explicit dependence on $p$. Hence, in the selection procedure we use the statistic $\widehat{\Theta}_j^\infty=\widehat{\Theta}_j \bra{1} =  |\hat \beta_{j,1}| + \cdots + |{\hat \beta}_{j,K_j}|,$ which will be compared to $B^{dj}n^{-1/2}$. Further details on the threshold will also be given in the next section. Under this assumption, we obtain the following result.
\begin{theorem}\label{thm:infty}
	Let $\fest$ by given by Eq.~\eqref{eqn:fest}. Given $r\in \bra{1, \infty}$, for any $d < s < r+1$, there exists $C>0$ such that 
	\begin{equation*}
		\sup_{f \in \besovgen{r}{q}{s}} \Ex \bra{\| {\hat f}_n - f\|_{L^\infty\bra{\sphere}}}\leq Cn^{-\frac{s-d}{2s+d}}.
	\end{equation*} 
\end{theorem} 
\begin{remark} \label{rem:remark}
As far as optimality is concerned, Eq.~\eqref{eqn:main} given in Theorem~\ref{thm:main}, achieves the same optimal minimax rates provided on $\reals$ by Kerkyacharian et al. \cite{kpt96}, where the established optimal rate of convergence was given by $n^{-sp/(2p+1)}$. Moreover, this rate is consistent with the results provided over the so-called regular zone by \cite{bkmpAoSb, durastanti2,dmg,monnier} for local and block thresholding estimates by needlets on the sphere, where the rates are nearly optimal due to the presence of a logarithmic term.

Regarding the $L^\infty$-risk function, according to Theorem \ref{thm:infty}, the rate established  is not optimal; see, e.g., Baldi et al. \cite{bkmpAoSb}. In the global thresholding paradigm, a straightforward generalization of the thresholding function $\Thetaest$ given by Eq.~\eqref{eqn:threshold} is not available (see Remark \ref{rem:infty}). Therefore, the upper bound for the case $p=\infty$ is established in a different framework, which can be reasonably assumed to cause the lack of optimality.   
\end{remark}

\subsection{The threshold}\label{sub:thres}

The construction of the threshold function $\tau_j$ is strictly related to Efromovich \cite{efrom} and Kerkyacharian et al. \cite{kpt96}, where analogous results were established in the real case. Let
\begin{equation*}\label{eqn:thetaj}
\Theta_j\bra{p}=\sum_{k=1}^{K_j}\abs{\betac}^p.
\end{equation*}
Using Eq.~\eqref{eqn:bbes}, it follows immediately that, if $f \in \besovgen{p}{q}{s}$,  
\begin{equation*}
\Theta_j\bra{p} \leq C B^{-jp\{s+d(\frac{1}{2}-\frac{1}{p})\}}.
\end{equation*}
Consider now the optimal bandwidth selection $J_s$. If $j\leq J_s$, 
\begin{equation*}
B^{-jp\{s+d(\frac{1}{2}-\frac{1}{p})\}}\geq \frac{B^{dj}}{n^{p/2}}.
\end{equation*}
Thus, even if $j\leq J_s$ doesn't imply $\Theta_j\bra{p} > B^{dj}/n^{{p}/{2}}$, according to \cite{efrom,kpt96}, one has that  $\Theta_j\bra{p}\geq B^{dj}/n^{{p}/{2}}$ implies $j\leq J_s$.

Clearly, the case $\bbra{\Theta_j\bra{p}\leq B^{dj}/n^{{p}/{2}}, j \leq J_s}$ provides no guarantee of a better performance if compared to the linear estimate, whose error is of order $B^{dj}/n^{{p}/{2}}$; see H\"ardle et al. \cite{WASA}. Thus, the natural choice is to construct a threshold function that keeps the level $j$ if and only if
\begin{equation*}
\Theta_j \bra{p}\geq \frac{B^{dj}}{n^{{p}/{2}}}.
\end{equation*}  

As pointed out in \cite{efrom, kpt96}, the natural estimator $|{\hat \beta}_{j,1}|^p + \cdots + |{\hat \beta}_{j,K_j}|^p$ for $\Theta_j\bra{p}$ yields an extremely large bias term and, therefore, undersmoothing effects. Hence, following the procedure as suggested in \cite{efrom,kpt96} (see also \cite{WASA}) but properly adapted to the needlet framework (see Lemma \ref{lemma:needloc}), we propose an alternative method as described below.

Let $p\in \mathbb{N}$ be even, $\Sigma_p$ denoting the set of $p$-dimensional vectors chosen in $\bbra{1,\dots,n}^p$ and also let $\upsilon=\left(i_1,\dots,i_p\right)$ be the generic element belonging to $\Sigma_p$. Define the U-statistic $\Thetaest$ by
\begin{align}
\label{eqn:threshold}
\Thetaest = \frac{1}{\binom {n} {p}} \sum_{k=1}^{K_j} \sum_{\upsilon \in \Sigma_p} \Psi_{j,k}^{\otimes p}\bra{X_{\upsilon},\varepsilon_\upsilon},
\end{align}
where
\begin{align*}\label{eqn:capitalpsi}
\Psi_{j,k}^{\otimes p}\bra{X_{\upsilon},\varepsilon_\upsilon} =  \prod_{h=1}^p \left\{f\bra{X_{i_h}}+\varepsilon_{i_h}\right\} \needlet{X_{i_h}}.
\end{align*}
Given that the sets of variables $\Xset$ and $\noiseset$ are independent, it can be easily seen that   
\begin{eqnarray}
\Ex\{ {\hat \Theta}_j(p)\}
\label{eqn:exptheta} = \sum_{k=1}^{K_j} \abs{\betac}^p= \Theta_j.
\end{eqnarray}

\begin{remark}\label{rem:infty}
As mentioned in the previous section, if we consider the case $p=\infty$, we lack a straightforward extension of $\Thetaest$. Hence, we choose 
	\begin{equation*}
\widehat{\Theta}_j^\infty=\sum_{k=1}^{K_j} |{\hat \beta}_{j,k}|.
	\end{equation*} 
so that the threshold function is given by 	
\begin{equation*}
\tau_j=\ind{\Theta_j ^\infty \geq \frac{B^{dj}}{n^{{1}/{2}}}}.
\end{equation*}  
\end{remark}

Our purpose is to establish two probabilistic bounds related to the $m$th moment and to the $m$th centered moment $\Thetaest$, respectively. We have that
\begin{align*}
\Ex [ \{ {\hat \Theta}_j(p)\}^m] &=\binom{n}{p}^{-m} \sum_{\upsilon_1,\dots,\upsilon_m \in \Sigma_p} \Ex\Bigl\{\sum_{k{_1},\ldots,k_m} \prod_{\ll=1}^m\Psi_{j,k_{\ll}}^{\otimes p} \bra{X_{\upsilon_\ll},\varepsilon_{\upsilon_\ll}} \Bigr\}.
\end{align*}
For any fixed configuration $\upsilon_1,\ldots,\upsilon_m$, let the sequence $c_1, \ldots, c_m$ denote the cardinality of the $\Ex\bra{[\{f\bra{X}+\epsilon\}\psi_{j,k}\bra{X}]^\ll}$ of size $\ll$ appearing in $\Ex\{\Thetaest\}$. Observe that 
\begin{equation*}\label{eqn:celle}
\sum_{\ll=1}^m \ll c_\ll=mp.
\end{equation*} 
Following Kerkyacharian et al. \cite{kpt96}, the next results hold.
\begin{proposition}\label{prop:thres}
Let $\Thetaest$ be given by Eq. \eqref{eqn:threshold}. Also, let $p$ be an even integer. Then, for $m\in \mathbb{N}$, there exists $\tilde{C_1}$ such that 
\begin{equation*}\label{eqn:mbound}
\Ex [ \{ {\hat \Theta}_j(p)\}^m] \leq  \tilde{C}_1 \sum_{\bra{c_1,\ldots, c_m} \in \Gamma_{m,p}}\bra{\frac{B^{jd}}{n}}^{\bra{\frac{mp}{2}-\sum_{\ell=1}^m c_\ell}} \frac{B^{-j\{c_1 \bra{s-\frac{d}{p}}-d\bra{1-\gamma}\}}}{n^{{mp}/{2}}}, 
\end{equation*}
where $\Gamma_{m,p}=\bbra{{c_1,\ldots,c_m} : \sum_{\ll=1}^m \ll c_\ll =mp}$.
\end{proposition}

\begin{proposition}\label{prop:threscent}
	Let $\Thetaest$ and $\Theta_j$ be given by \eqref{eqn:threshold} and \eqref{eqn:exptheta} respectively.  Also, let $p,m$ be even integers. Then, there exists $\tilde{C_2}$ such that 
	\begin{equation*}\label{eqn:mboundcent}
	{\rm E} \{ | {\hat \Theta}_j(p) - \Theta_j (p)|^m\}  \leq  \tilde{C_2}  \sum_{h=1}^p \bra{\frac{B^{jd}}{n^{{p}/{2}}}}^{\frac{mh}{p}}\{ \Theta_j\bra{p}\}^{\bra{1-{h}/{p}}m}. 
	\end{equation*}
\end{proposition}

\begin{remark}\label{remark:interpole} According to \cite{kpt96}, this procedure can be easily extended to the case of $p$ is not being an even natural number, by means of an interpolation method. Indeed, by fixing $p_1,p_2 \in \mathbb{N}$, both even, we can rewrite $p=\delta p_1 + \bra{1-\delta}p_2$, as 
\begin{equation*}
\Thetaest=\{\hat{\Theta}_j\bra{p_1}\}^{\delta}\{\hat{\Theta}_j\bra{p_2}\}^{1-\delta}.
\end{equation*} 

The following lemma is crucial for the application of our interpolation method. As in \cite{kpt96}, we consider for the sake of simplicity just the case $p_2-p_1=2$, so that $p_0\leq p\leq p_0^\prime$, $p_0^\prime=p_0+2$.
\end{remark}
\begin{lemma}\label{lemma:central}
For any even $m \in \mathbb{N}$, 
\begin{equation*}
	\Ex\bbra{\abs{\frac{1}{n}\hat{\Theta}_j\bra{p_0} - \Theta_j\bra{p_0^\prime}}^m} \leq  \tilde{C_2}  \sum_{h=1}^{p_0^\prime} \{ \Theta_j\bra{p_0^\prime-h}\}^{m}n^{- {mh}/{2}}. 
\end{equation*}	
\end{lemma}

We conclude this section with a result regarding the behavior of $\widehat{\Theta}_j^\infty$. 

\begin{proposition}\label{prop:threspuno}
	Let $\widehat{\Theta}_j^\infty$ and $\Theta^\infty_j$ be given by \eqref{eqn:threshold} and \eqref{eqn:exptheta} for $p=1$, respectively. Then, there exists $\tilde{C}_\infty$ such that  
	\begin{equation*}
	\Ex ( |\widehat{\Theta}_j^\infty - \Theta^\infty_j|^2) \leq  \tilde{C}_\infty  B^{dj}/n. 
	\end{equation*}
\end{proposition}

\section{Simulations}\label{sec:numerical}

In this section, we present the results of some numerical experiments performed over the unit circle $\mathbb{S}^1$. In particular, we are mainly concerned with the empirical evaluation of $L^2$-risks obtained by global thresholding techniques, which are then compared to the $L^2$-loss functions for linear wavelet estimators. 

As in any practical situation, the simulations are computed over finite samples and therefore, they can be considered as a reasonable hint. Furthermore, they can be viewed as a preliminary study to practical applications on real data concerning estimation of the power spectrum of the Cosmic Microwave Background radiation; see, e.g., Fa\"y et al. \cite{fay08}.
 
The needlets over the circle used here are based on a weight function $b$ which, analogously to Baldi et al. \cite{bkmpAoSb}, is a properly rescaled primitive of the function 
\begin{equation*}
x\mapsto e^{-1/(1-x^2)}; 
\end{equation*}
see also Marinucci and Peccati \cite{MaPeCUP}. Following Theorem~\ref{thm:main}, we fix $B=2$ and $n=2^6,2^7,2^8$ and $J_n=6,7,8$, respectively. The $U$-statistic $\widehat{\Theta}_j\bra{2}$, corresponding to the $L^2$ risk considered here, results in considerable computational effort, because it is built over 2,016, 8,128 or 32,640 possible combinations of needlet empirical coefficients for $J_n=6,7,8$, respectively.
 
By choosing a test function $F$ and fixing the set of locations $X_1,\ldots,X_n$, we obtain the corresponding $Y_1,\ldots,Y_n$ by adding to $F\bra{X_i}$ a Gaussian noise, with three different amplitudes, i.e., the noise standard deviation $\sigma_\varepsilon$ is chosen to be equal to $0.25M$, $0.5M$ or $0.75M$, where $M$ is the $L^\infty$-norm of the test function. Therefore, the following three numerical experiments are performed.

\begin{example}
According to Baldi et al. \cite{bkmpAoSb}, we use the uniform test function defined, for all $x\in \mathbb{S}^1$, by
\begin{equation*}
F_1\bra{x}=\frac{1}{4\pi} \,.
\end{equation*}	
In this case, for every $j,k$, we get $\betac=0$. The performance of our model can be roughly evaluated by simply controlling how many resolution levels pass the selection procedure. For all the choices of $n$ and $\sigma$, we get $\tauthres=0$ for all $j \in \{ 1,\ldots,J_n\}$. On the other hand, we consider a finite number of resolution levels and therefore of frequencies. Thus, it is possible that higher resolution levels, involving higher frequencies, could be selected by the thresholding procedure.  
\end{example} 

\begin{example}
In this example, we choose the function defined, for all $x \in \mathbb{S}^1$, by
	\begin{equation*}
	F_2\bra{x}=\cos\l(4x\r).
	\end{equation*}
In this case, the test function corresponds to the real part of one of the Fourier modes over the circle (with eigenvalue $4$). This choice allows us to establish whether the thresholding procedure is able to select only the suitable corresponding resolution levels, as the amplitude of the noise increases. As expected, for every $n$ and for every $\sigma_\varepsilon$, we get $\tauthres=1$ for $j=2$ (containing the frequency $k=4$) and $0$ otherwise. Table \ref{tab:1} presents the value for $L^2$-risks for different values of $J_n$ and $\sigma_\varepsilon$, while Figure \ref{fig:ex42} illustrates the graphical results for the case $J_n=8$ and $\sigma_\varepsilon=0.5 M$.\\
\begin{table}
	\centering
\begin{tabular}{|c||c|c|c || c |c| c| }	
	\hline
	{Example 4.2} & \multicolumn{3}{|c||}{Global} & 	\multicolumn{3}{|c|}{Linear} \\ \hline
	  $J_n \downarrow  \backslash  \sigma_\varepsilon \rightarrow$ & $0.25 M $&$0.50M$&$0.75M$& $0.25M $&$0.50M$&$0.75M$\\ \hline	  
	$6$& 7.82 &  65.29 & 108.46 & 80.23 & 411.38 &889.15 \\
	$7$& 1.90 & 9.38 & 67.09 & 26.75 & 141.41 &451.61 \\
	$8$& 1.77 &  18.88 & 54.03 & 36.69 & 96.82& 434.80 \\
	\hline
\end{tabular}
\caption{Example 4.2 - Values for $L^2$ risk}\label{tab:1} 
\end{table}      
\begin{figure}[htbp]
	\centering
\fbox{\includegraphics[scale=0.43
	]{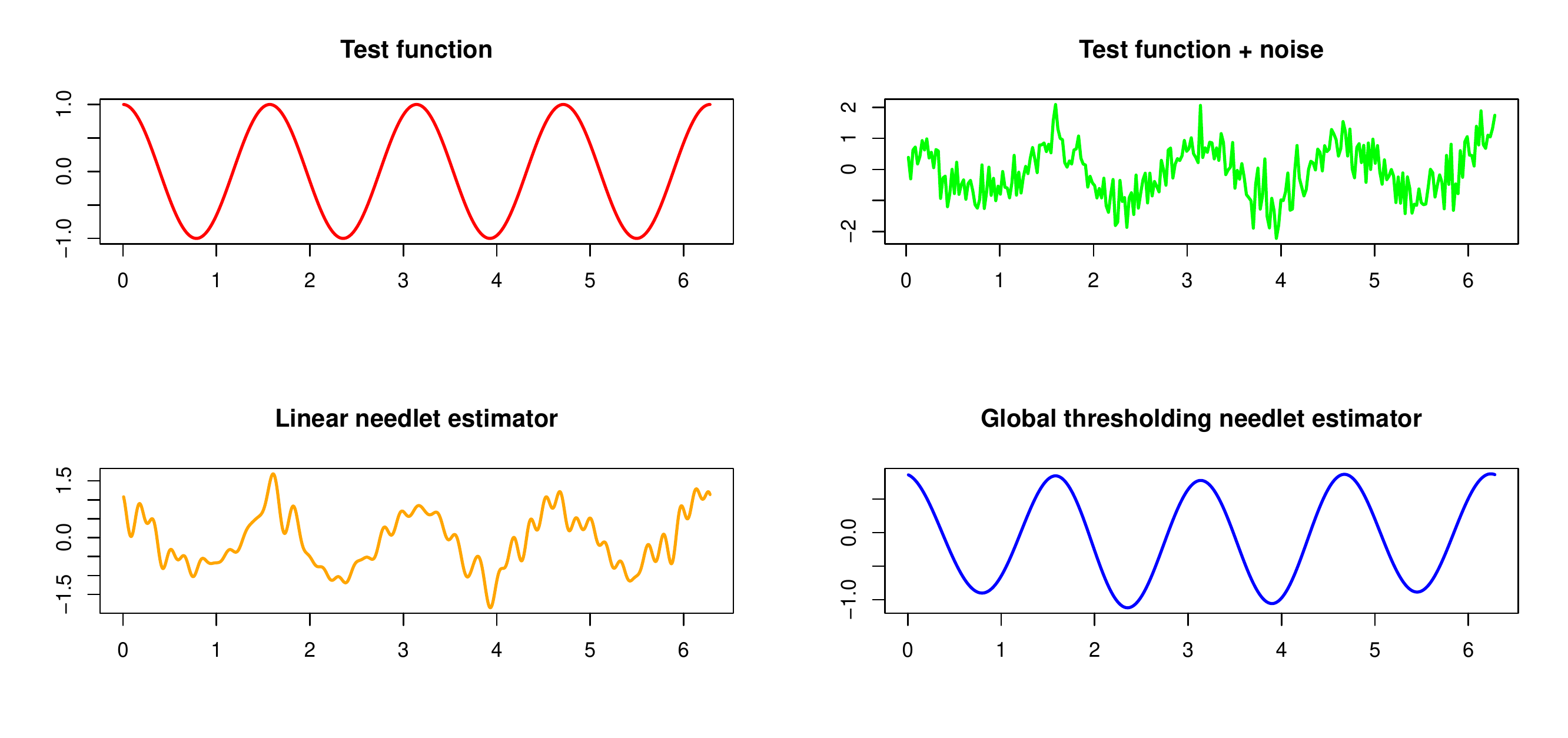}}
	\caption{Example 4.2 --- $J_n=8$, $\sigma_\varepsilon=0.5 M$}\label{fig:ex42}
\end{figure}
\end{example} 
 
\begin{example}
A more general function is chosen here, which is defined, for all $x \in \mathbb{S}^1$, by
\begin{equation*}
F_3\bra{x}= \left\{ e^{-(x-3\pi/2)^2}+2e^{-\l(x-2\r)^2}\right\} \sin\l(-2x\r),
\end{equation*}
depending on a larger set of Fourier modes. In this case, Table \ref{tab:2} gives the values of $L^2$-risks corresponding to different $J_n$ and $\sigma_\varepsilon$, while Figure \ref{fig:ex43} presents the graphical results for the case $J_n=8$ and $\sigma_\varepsilon=0.5 M$. Table \ref{tab:3} contains, for every pair $J_n$, $\sigma_\varepsilon$, the resolution levels selected by the procedure.\\
\begin{table}
	\centering
	\begin{tabular}{|c||c|c|c || c |c| c| }	
		\hline
		{Example 4.3} & \multicolumn{3}{|c||}{Global} & 	\multicolumn{3}{|c|}{Linear} \\ \hline
		$J_n \downarrow  \backslash  \sigma_\varepsilon \rightarrow$ & $0.25 M $&$0.50M$&$0.75M$& $0.25M $&$0.50M$&$0.75M$\\ \hline	  
		$6$& 62.01 & 208.94 & 625.02 & 227.32 & 1372.64&2296.28 \\
		$7$& 58.87 & 150.73 & 277.98 & 150.94 & 644.31&1294.59 \\
		$8$& 51.14 & 40.65 & NA & 321.05 & 1271.31&NA \\
		\hline
	\end{tabular}
	\caption{ Example 4.3 --- Values for $L^2$ risk}	\label{tab:2} 
\end{table}     

\begin{center}   
\begin{table}
\centering
	\begin{tabular}{|c||c|c|c | }		\hline
		{Example 4.3} & \multicolumn{3}{||c|}{Selected $j$}  \\ \hline
		$J_n \downarrow  \backslash  \sigma_\varepsilon \rightarrow$ & $0.25 M $&$0.50M$&$0.75M$ \\ \hline	  
		$6$& 1 & 1,2 & 1,2,3 \\
		$7$&1 & 1,4 & 1,2  \\
		$8$& 1 & 1,2 & 1,3,5  \\
		\hline
	\end{tabular}
	\caption{Example 4.3 - Values of the function $\tauthres$}		\label{tab:3} 	
\end{table}   
\end{center}

\begin{figure}[htbp]
	\centering
	\fbox{\includegraphics[scale=0.43
		]{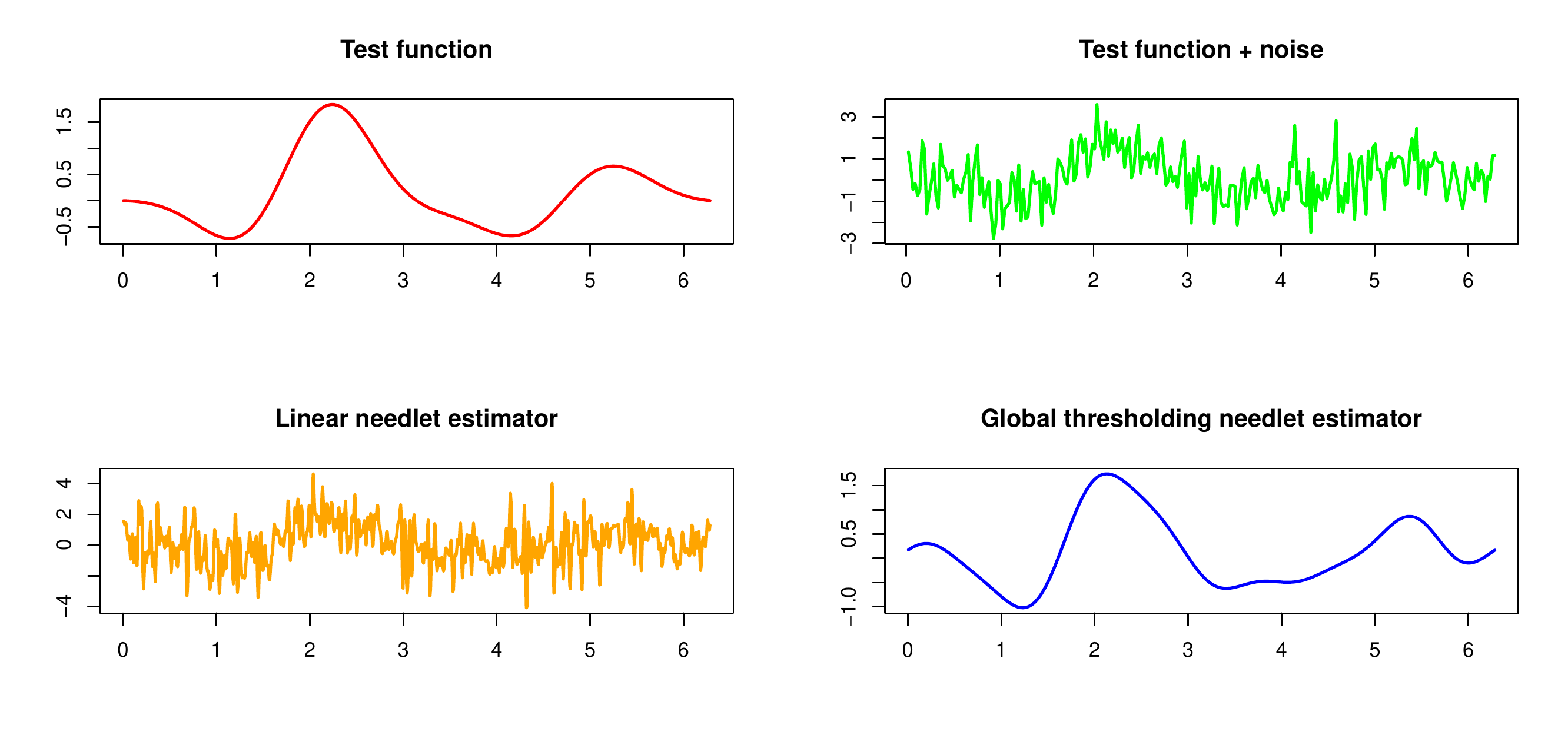}}
	\caption{Example 4.3 - $J_n=8$, $\sigma_\varepsilon=0.5 M$}\label{fig:ex43}
\end{figure}
\end{example}

\newpage
\section{Proofs}\label{sec:proofs}

In this section, we provide proofs for the main and auxiliary results. 

\subsection{Proof of the main results}   

The proofs of Theorem \ref{thm:main} and Theorem \ref{thm:infty} follow along the same lines as the proof of Theorem 8 in Baldi et al. \cite{bkmpAoSb}. 

\begin{proof}[Proof of Theorem \ref{thm:main}]

Following, for instance, \cite{bkmpAoSb, donoho1,durastanti2,dmg, WASA, kpt96} and as mentioned in Section~\ref{sub:est}, the $L^p$-risk $\Ex ( \|\fest-f\|_\Lp^p )$ can be decomposed as the sum of a stochastic and a bias term. More specifically,
	\begin{align*}
	\Ex (\| {\hat f}_n - f\|_{\Lp}^p) \leq  2^{p-1}&\bbra{\Ex\bra{\norm{\sum_{j=0}^{J_n} \sum_{k=1}^{K_j} (\tauthres\betaest-\betac)\psi_{j,k} }_\Lp^p} \right.&\\
	 & \ \ \ \left.+\norm{ \sum_{j>J_n} \sum_{k=1}^{K_j} \betac \psi_{j,k} }_\Lp^p}.  
	\end{align*}
Using the definition of Besov spaces, we obtain for the bias term the following inequality:
	\begin{align}
	\norm{ \sum_{j>J_n} \sum_{k=1}^{K_j} \betac \psi_{j,k} }_\Lp^p \leq & \notag   
	 \sum_{j>J_n}\norm{ \sum_{k=1}^{K_j} \betac \psi_{j,k} }_\Lp^p  \notag \leq  CB^{-s p J_n} \leq  Cn^{-\frac{ps}{2s+d}}.\label{eqn:bias}
	\end{align}
Following Baldi et al. \cite{bkmpAoSb}, the stochastic term can be split into four terms, i.e.,
	\begin{equation*}
	\Ex \bra{\norm{\sum_{j=0}^{J_n} \sum_{k=1}^{K_j} (\tauthres\betaest-\betac)\psi_{j,k} }_\Lp^p} \leq 4^{p-1} (Aa + Au + Ua + Ua) , 
	\end{equation*} 
	where
	\begin{align*}
	Aa = 	\Ex\!\bra{\!\norm{\!\sum_{j=0}^{J_n}\! \sum_{k=1}^{K_j} (\tauthres\betaest -\betac) \psi_{j,k} \ind{\Thetaest\geq \frac{B^{dj}}{n^{{p}/{2}}}}\ind{\Theta_j\bra{p}\geq \frac{B^{dj}}{2n^{{p}/{2}}}} }_\Lp^p},
	\end{align*}
	\begin{align*}
	Au =  	\Ex\!\bra{\!\norm{\!\sum_{j=0}^{J_n} \sum_{k=1}^{K_j} (\tauthres\betaest-\betac)\psi_{j,k} \ind{\Thetaest\geq \frac{B^{dj}}{n^{{p}/{2}}}}\ind{\Theta_j\bra{p}\leq \frac{B^{dj}}{2n^{{p}/{2}}}}\!} _\Lp^p},
	\end{align*}
	\begin{align*}
 	Ua =  	\Ex\!\bra{\!\norm{\!\sum_{j=0}^{J_n} \sum_{k=1}^{K_j} (\tauthres\betaest-\betac) \psi_{j,k} \ind{\Thetaest\leq \frac{B^{dj}}{n^{{p}/{2}}}}\ind{\Theta_j\bra{p}\geq \frac{2B^{dj}}{n^{{p}/{2}}}} \!}_\Lp^p},
	\end{align*}
	\begin{align*}
		Uu = \Ex\!\bra{\!\norm{\!\sum_{j=0}^{J_n} \sum_{k=1}^{K_j} (\!\tauthres\betaest\!-\!\betac) \psi_{j,k} \ind{\Thetaest\leq \frac{B^{dj}}{n^{{p}/{2}}}}\ind{\Theta_j\bra{p}\leq \frac{2B^{dj}}{n^{{p}/{2}}}}\!} _\Lp^p}.
	\end{align*}

Following Durastanti et al. \cite{dmg}, the labels $A$ and $U$ denote the regions where
$\Thetaest$ is larger and smaller than the threshold $B^{dj}n^{-p/2}$ respectively, whereas $a$ and $u$ refer to the regions where the deterministic $\Theta_j\bra{p}$ are above and under a new threshold, given by $2^{-1}B^{dj}n^{-p/2}$ for $a$ and $2B^{dj}n^{-p/2}$ for $u$. The decay of $Aa$ and $Uu$ depends on the properties of Besov spaces, while the bounds on $Au$ and $Ua$ depend on the probabilistic inequalities concerning $\betaest$ and $\Thetaest$, given in Propositions \ref{prop:betac}, \ref{prop:thres} and \ref{prop:threscent}.

Let $p \in \mathbb{N}$ be even. Then, using the definition of the optimal bandwidth selection, we have
	\begin{eqnarray*}
	Aa &\leq &  C_1^{\prime} \bra{J_n+1}^{p-1} \!\sum_{j=0}^{J_n}  \norm{\psi_{j,k}}_\Lp^p \sum_{k=1}^{K_j}\!\Ex\bra{|{\hat \beta}_{j,k} - \beta_{j,k}|^p} \ind{\abs{\Theta_j\bra{p}\geq \frac{B^{dj}}{2n^{{p/2}}}}}\\
	&\leq & C_1^{\prime\prime} \left\{\sum_{j=0}^{J_s} B^{jd{p/2}} n^{-{p}/{2}}+n^{-{p/2}}\sum_{j=J_s}^{J_n} \sum_{k=1}^{K_j} B^{dj\bra{{p/2}-1}} \frac{\Theta_j\bra{p}}{B^{dj}n^{-{p/2}}}\right\}\\
	&\leq & C_1^{\prime\prime\prime} \left\{B^{J_sd{p/2}} n^{-{p}/{2}} + \sum_{j=J_s}^{J_n} \bra{\sum_{k=1}^{K_j}\abs{\betac}^p\norm{\psi_{j,k}}_{\Lp}^p} \right\} \\
	&\leq & C_1 \bra{ B^{J_s d{p/2}} n^{-{p}/{2}}+B^{-J_s sp}}\\
	&= & C_1 n^{-\frac{sp}{2s+d}},
	\end{eqnarray*}
given that 
	\begin{equation}\label{eqn:bjs}
	B^{J_s d{p/2}} n^{-{p}/{2}}=n^{\frac{dp}{2\bra{2s+d}}-{p}/{2}}=n^{-\frac{sp}{2s+d}}.
	\end{equation}

Similarly, for the region $Au$ we obtain 
	\begin{align*}
	Au  \leq C_2^\prime \bra{Au_1 + Au_2},
	\end{align*}		
	where
	\begin{align*}		
    Au_1	= &\sum_{j=0}^{J_s} \sum_{k=1}^{K_j}\norm{\psi_{j,k}}_{\Lp}^p \Ex (|{\hat \beta}_{j,k} - \beta_{j,k}|^p),\\ 
    Au_2    = &\sum_{j=J_s}^{J_n} \sum_{k=1}^{K_j}\norm{\psi_{j,k}}_{\Lp}^p \Ex \left[ |{\hat \beta}_{j,k} - \beta_{j,k}|^p\ind{\Thetaest\geq \frac{B^dj}{n^{-{p}/{2}}}}\right]. 
	\end{align*}

Using Eq.s~\eqref{eqn:Lpnorm} and \eqref{eqn:bjs}, it is easy to see that
	\begin{equation*}
	   Au_1 \leq C_2 n^{-\frac{sp}{2s+d}}.
	\end{equation*}
Regarding $Au_2$, using H\"older inequality with $1/\alpha^\prime+1/\alpha=1$, the generalized Markov inequality with even $m \in \mathbb{N}$ and Proposition \ref{prop:thres}, we obtain
	\begin{eqnarray*}
	Au_2 &\leq & C\sum_{j=J_s}^{J_n}\sum_{k=1}^{K_j} \norm{\psi_{j,k}}_\Lp^p \Ex\left[ |{\hat \beta}_{j,k} - \beta_{j,k}|^p\ind{\Thetaest\geq \frac{B^{dj}}{n^{{p/2}}}} \right]\\
	&\leq & C \sum_{j=J_s}^{J_n}\sum_{k=1}^{K_j} B^{jd\bra{{p/2}-1}}\! \Bigl\{ \Ex\bra{|{\hat \beta}_{j,k} - \beta_{j,k}|^{p\alpha^\prime}}\Bigr\}^{1/\alpha^\prime}\left[ \Pr\left[\left\{\Thetaest\geq \frac{B^{dj}}{n^{{p/2}}} \right\} \right]\right]^{1/\alpha}\\
	&\leq & C \sum_{j=J_s}^{J_n} B^{jd{p/2}} n^{-{p/2}}\left[\frac{\Ex \Bigl[\Bigl\{\Thetaest^m\Bigr\} \Bigr]}{\bra{B^{dj}/n^{{p/2}}}^m} \right]^{{1}/{\alpha}}\\
	&\leq & C \sum_{j=J_s}^{J_n}\!\!\! \frac{B^{j\frac{dp}{2}}}{ n^{{p/2}}} \, \left[\sum_{\!\bra{c_1,\ldots, c_m} \in \Gamma_{m,p}}\! \!\! \bra{\frac{B^{jd}}{n}}^{\bra{\frac{mp}{2}-\sum_{\ell=1}^m c_\ell}}\!\!\! \!\!\!\!\!B^{-j\{c_1 \bra{s-\frac{d}{p}}\}}B^{dj\{\bra{1-m}-\gamma\}}\! 
	\right]^{{1}/{\alpha}} \\
	&\leq& Au_{2,1}+Au_{2,2}, 
	\end{eqnarray*}
	where $Au_{2,1}$ and $Au_{2,2}$ are defined by splitting $\Gamma_{m,p}$ into two subsets, $\Gamma_{m,p}^{\bra{+}}$ and $\Gamma_{m,p}^{\bra{-}}$. These subsets are defined as 
	\begin{align*} 
		\Gamma_{m,p}^{\bra{+}}& :=\bbra{c_1,\ldots,c_m:\frac{mp}{2}-\sum_{i=i}^m c_m\geq 0};\\ \Gamma_{m,p}^{\bra{-}}&:=\bbra{c_1,\ldots,c_m:\frac{mp}{2}-\sum_{i=i}^m c_m\leq 0}.
	\end{align*} 
Note that $1-\gamma\in \sbra{0,1}$. Hence we choose $m,\alpha$ so that $m>1+{\alpha p}/{2}$. It can be easily verified that 
	\begin{eqnarray*}
	Au_{2,1} &\leq & C^\prime \sum_{j=J_s}^{J_n} B^{j\frac{dp}{2}} n^{-{p/2}}\Bigl[\sum_{\bra{c_1,\ldots, c_m}\in \Gamma_{m,p}^{\bra{+}} }B^{dj\{\bra{1-m}-\gamma\}} \Bigr]^{1/\alpha}\\
	&\leq & C^{\prime \prime} B^{J_s\frac{dp}{2}} n^{-{p/2}} \left\{ \sum_{\bra{c_1,\ldots, c_m}\in \Gamma_{m,p}^{\bra{+}} }B^{dJ_s\bra{1-m}} \right\}^{1/\alpha}\\
	&\leq & C^{\prime \prime \prime} n^{-\frac{sp}{2s+d}}
	\end{eqnarray*}
On the other hand, 
	\begin{eqnarray*}
	Au_{2,2} &\leq & \!C^{\prime}\!\!\!\!\!\!\!\!\sum_{\bra{c_1,\ldots, c_m}\in \Gamma_{m,p}^{\bra{-}}} \!\!\!\!\!\!\!\!\! B^{-J_s\frac{c_1\bra{s-{d}/{p}}}{\alpha}}\sum_{j=J_s}^{J_n}B^{j{dp}/{2}} n^{-{p/2}} \left\{{\!\bra{\frac{B^{jd}}{n}}^{{\bra{\frac{mp}{2}-\sum_{\ell=1}^m c_\ell}}/{\alpha}} \!B^{\frac{1-m-\gamma}{\alpha}dj}}\right\}\\
	& \leq & C^{\prime \prime}\bra{\frac{B^{dJ_s}}{n}}^{{p/2}}\!\!\!\!\!\!\!\!\sum_{\bra{c_1,\ldots, c_m}\in \Gamma_{m,p}^{\bra{-}}}\!\!\!\!\!\!\!\! B^{-J_s\frac{c_1}{\alpha}\bra{s-\frac{d}{p}}}{\bra{\frac{B^{J_sd}}{n}}^{\bra{\frac{mp}{2}-\sum_{\ell=1}^m c_\ell}/\alpha} B^{\frac{1-m-\gamma}{\alpha}dJ_s}}\\
	& \leq & C^{\prime \prime}\bra{\frac{B^{dJ_s}}{n}}^{{p/2}}\!\!\!\!\!\!\!\!\sum_{\bra{c_1,\ldots, c_m}\in \Gamma_{m,p}^{\bra{-}}} \!\!\!\!\!\!\!\!n^{-\frac{ c_1\bra{s-\frac{d}{p}}}{\alpha\bra{2s+d}}}{\{n^{\frac{2s}{\bra{2s+d}\alpha}\bra{\frac{mp}{2}-\sum_{\ell=1}^m c_\ell}}\} n^{\frac{d\bra{1-m-\gamma}}{\alpha\bra{2s+d}}}}\\
	& \leq & C^{\prime \prime \prime} n^{-\frac{sp}{2s+d}},
	\end{eqnarray*}
and therefore, 
	\begin{equation*}
	Au \leq C_2  n^{-\frac{sp}{2s+d}}. 
	\end{equation*}
	Consider now $Ua$. We have that
	\begin{eqnarray*}
	Ua &\leq & C_3^\prime \left(\sum_{j=0}^{J_s} \norm{\sum_{k=1}^{K_j}\betac\psi_{j,k}}_{\Lp}^p \Ex\left[ \ind{\Thetaest\leq \frac{B^{dj}}{n^{{p/2}}}}\ind{\Theta_j\geq 2\frac{B^{dj}}{n^{{p/2}}}}\right]\right. \\ & & \hspace{3cm} \left. +\sum_{j=J_s}^{J_n} \norm{\sum_{k=1}^{K_j}\betac\psi_{j,k}}_{\Lp}^p   \right)\\
	& = & Ua_1 + Ua_2.
	\end{eqnarray*}
	It is easy to see that 
	\begin{align*}
	Ua_2 \leq C_3^{\prime \prime} B^{-J_s ps} = C_3^{\prime \prime} n^{-\frac{-sp}{2s+d}}. 
	\end{align*}
	On the other hand, using the generalized Markov inequality, Proposition \ref{prop:threscent} with $m=p$ and Eq. \eqref{eqn:bjs}, we have that
	\begin{eqnarray*}
	Ua_{1} &=&\sum_{j=0}^{J_s}B^{dj\bra{{p/2}-1}}\Ex\left[\Theta_j \ind{\Thetaest\leq \frac{B^{dj}}{n^{{p/2}}}}\ind{\Theta_j\geq 2\frac{B^{dj}}{n^{{p/2}}}}\right]\\
	& \leq & C\sum_{j=0}^{J_s}B^{dj\bra{{p/2}-1}}\Theta_j \Pr\Bigl\{|{\hat \Theta}_j(p)-\Theta_j|\geq \frac{1}{2}\Theta_j \Bigr\}\ind{\Theta_j \geq 2\frac{B^{dj}}{n^{{p/2}}}}\\
	& \leq & C\sum_{j=0}^{J_s}B^{dj\bra{{p/2}-1}}\Theta_j^{1-m} \Ex\bra{|{\hat \Theta}_j(p)-\Theta_j|^m} \ind{\Theta_j \geq 2\frac{B^{dj}}{n^{{p/2}}}}\\
	& \leq & C\sum_{j=0}^{J_s}B^{dj\bra{{p/2}-1}}\Theta_j 
	\sum_{\ll=1}^p \bra{\Theta_j^{-1} \frac{B^{dj}}{n^{{p/2}}}}^{\frac{m\ell}{p}}\\
	& \leq & C\sum_{j=0}^{J_s}B^{dj\bra{{p/2}-1}}\frac{B^{dj}}{n^{{p/2}}}\\
	&\leq & C n^{\frac{-sp}{2s+d}} ,
	\end{eqnarray*}
and therefore,
	\begin{equation*}
Ua\leq C_3n^{-\frac{-sp}{2s+d}}. 
	\end{equation*}
Finally, in view of Eq.~\eqref{eqn:Lpnorm} and Eq.~\eqref{eqn:bjs}, we have that
\begin{eqnarray*}
Uu &\leq & C_4 ^\prime \left[\sum_{j=0}^{J_s} B^{jd\bra{{p/2}-1}}\Theta_j\bra{p}\ind{\Theta_j\leq 2\frac{B^{dj}}{n^{{p/2}}}} +\sum_{j=J_s}^{J_n} \sum_{k=1}^{K_j}\abs{\betac}^p \norm{\psi_{j,k}}_{\Lp}^p \right]\\
&\leq & C_4 ^{\prime \prime} \bra{\sum_{j=0}^{J_s} B^{jd{p/2}}n^{-{p/2}}+ \sum_{j=J_s}^{J_n}B^{-jsp}}\\
&\leq & C_4 n^{\frac{-sp}{2s+d}} .
\end{eqnarray*}

We now need to extend these results to any $p \in \bra{1,r}$ using the interpolation method described in Remark \ref{remark:interpole}. The two terms that have to be studied separately are $Au$ and $Ua$, in particular $Au_2$ and $Ua_1$, since they involve the probabilistic inequalities described in Propositions \ref{prop:thres} and \ref{prop:threscent}, holding only for even $p\in\mathbb{N}$.  According to \cite{kpt96}, the generalization in the case of $Au_2$ is obtained by bounding
\begin{equation*}
\Ex\left[\left\{\hat{\Theta}_j\bra{p}\right\}^m\right]\leq C \Ex \left[\left\{\hat{\Theta}_j\bra{p_1}\right\}^m\right]^\delta\Ex \left[\left\{\hat{\Theta}_j\bra{p_2}\right\}^m\right]^{1-\delta}
\end{equation*}  
Indeed,
\begin{eqnarray*}
Au_2 &\leq & C \sum_{j=J_s}^{J_n}\left[ B^{j\frac{dp_1}{2}} n^{-\frac{p_1}{2}}\!\left[ \sum_{\bra{c_1,\ldots, c_m} \in \Gamma_{m,p_1}}\!\bra{\frac{B^{jd}}{n}}^{\bra{\frac{mp_1}{2}-\sum_{\ell=1}^m c_\ell}} \right. \right.\\
&& \left. \left. \times B^{-j c_1 \bra{s-\frac{d}{p_1}}}B^{dj\bbra{\bra{1-m}-\gamma}} 
\right]^{1/\alpha}\right]^{\delta} \left[ B^{j\frac{dp_2}{2}} n^{-\frac{p_2}{2}} \right.\!\\
&\times &\! \! \! \!\!\left.\sbra{\sum_{\bra{c_1,\ldots, c_m} \in \Gamma_{m,p_2}}\!\!\!\!\!\!\!\!\!\bra{\frac{B^{jd}}{n}}^{\bra{\frac{mp_2}{2}-\sum_{\ell=1}^m c_\ell}} \!\!\!\!\!B^{-j c_1 \bra{s-\frac{d}{p_2}}}B^{dj\bbra{\bra{1-m}-\gamma}} 
	}^{1/\alpha}\right]^{1-\delta}.
\end{eqnarray*}
The result above follows from Eq.s \eqref{eqn:besovmin} and \eqref{eqn:besovmax}, so that 
\begin{align*}
\besovgen{p}{q}{s}\subset &\besovgen{p_1}{q}{s}; \besovgen{p}{q}{s}\subset \besovgen{p_2}{q}{s-d\bra{\frac{1}{p}-\frac{1}{p_2}}}.
\end{align*}
Straightforward calculations lead to the claimed result. On the other hand, in order to study $Ua_1$, we apply Lemma \ref{lemma:central} to obtain
\begin{align*}
Ua_1 \leq &C\sum_{j=0}^{J_s} B^{dj\bra{{p/2}-1}}\Ex\sbra{\Theta_j \ind{\Thetaest\leq \frac{B^{dj}}{n^{{p/2}}}}\ind{\Theta_j\bra{p}\geq 2\frac{B^{dj}}{n^{{p/2}}}}}\\
\leq &C\sum_{j=0}^{J_s} B^{dj\bra{{p/2}-1}}\Ex\left[\Theta_j \ind{\hat{\Theta}_j\bra{p_0}\leq \frac{B^{dj}}{n^{\frac{p_0}{2}}}}\ind{\Theta_j\bra{p}\geq 2\frac{B^{dj}}{n^{{p/2}}}}\right.\\
& \left.\ind{\hat{\Theta}_j\bra{p_0^\prime}\leq \frac{B^{dj}}{n^{\frac{p_0^\prime}{2}}}}\ind{\Theta_j\bra{p}\geq 2\frac{B^{dj}}{n^{{p/2}}}}\right]\\
\leq& C\sum_{j=0}^{J_s} B^{dj\bra{{p/2}-1}} \left[\Pr\bbra{\abs{\hat{\Theta}_j\bra{p_0^\prime}-\Theta_j\bra{p_0^\prime}}\geq\frac{\Theta_j\bra{p_0^\prime}}{2}}\right.\\
&+\left.\Pr\bbra{\abs{\frac{1}{n}\hat{\Theta}_j\bra{p_0}-\Theta_j\bra{p_0^\prime}}\geq\frac{\Theta_j\bra{p_0^\prime}}{2}}\right],
\end{align*}
because $\bbra{\Theta_j\bra{p}\geq 2 \frac{B^{dj}}{n^{{p/2}}}} \subset \bbra{\Theta_j\bra{p^\prime_0}\geq 2 \frac{B^{dj}}{n^{\frac{p_0^\prime}{2}}}}$. Finally, by applying Markov inequality with $m>p$, $m \in \mathbb{N}$ even, we have that
\begin{align*}
Ua_1 \leq & C\sum_{j=0}^{J_s} B^{dj\bra{{p/2}-1}}  \sum_{h=1}^{p_0^\prime} \bbra{ \Theta_j\bra{p_0^\prime-h}}^{m}n^{-\frac{mh}{2}}\\
&\times\bbra{\Theta_j\bra{p_0^\prime-h}}^{1-m}\ind{\Theta_j\bra{p_0^\prime}\geq 2\frac{B^{dj}}{n^{\frac{p_0^\prime}{2}}}}\\
\leq & C\sum_{j=0}^{J_s} B^{dj\bra{{p/2}-1}}  \sum_{h=1}^{p_0^\prime} \bbra{ \Theta_j\bra{p_0^\prime}}^{\frac{p-mh}{{p_0^\prime}}}n^{-\frac{mh}{2}}\\
&\times B^{dj\bra{1-\frac{p}{p_0^\prime} + \frac{mh}{p_0^\prime} }} \ind{\Theta_j\bra{p_0^\prime}\geq 2\frac{B^{dj}}{n^{\frac{p_0^\prime}{2}}}}\\
\leq & C\sum_{j=0}^{J_s} B^{dj\bra{{p/2}-1}}  \sum_{h=1}^{p_0^\prime} \bra{ \frac{B^{dj}}{n^{\frac{p_0^\prime}{2}}}}^{\frac{p-mh}{{p_0^\prime}}}n^{-\frac{mh}{2}} B^{dj\bra{1-\frac{p}{p_0^\prime} + \frac{mh}{p_0^\prime} }}\\
\leq & C n^{-\frac{sp}{2s+d}}
\end{align*}
\end{proof}
\begin{proof}[Proof of Theorem \ref{thm:infty}]
Similarly to the previous proof, note that 
\begin{align*}
\Ex\bra{\| {\hat f}_n - f\|_{L^\infty\bra{\sphere}}} \leq &C \l\{ \Ex\bra{\norm{\sum_{j=0}^{J_n}\sum_{k=1}^{K_j}\bra{\tau_j \betaest-\betac}\psi_{j,k}}_{L^{\infty}\bra{\sphere}}}\r.\\
& \l.+ \norm{\sum_{j>J_n}\sum_{k=1}^{K_j}\bra{\betac}\psi_{j,k}}_{L^{\infty}\bra{\sphere}}\r\}.
\end{align*}
If $f \in \besovgen{\infty}{\infty}{s}$, $\abs{\betac}\leq M B^{-j\bra{s+\frac{d}{2}}}$ for any $k=1,\ldots, K_j$, then 
by Eq.s \eqref{eqn:Lpnorm} and \eqref{eqn:boundsuppa} with $p=\infty$, we get
\begin{align*}
 \norm{\sum_{j>J_n}\sum_{k=1}^{K_j}\betac\psi_{j,k}}_{L^{\infty}\bra{\sphere}}&\leq \sum_{j>J_n}\norm{\sum_{k=1}^{K_j}\betac\psi_{j,k}}_{L^{\infty}\bra{\sphere}} \\&\leq C \sum_{j>J_n} \sup_{k=1,\ldots,K_j}\abs{\betac} \norm{\psi_{j,k}}_{L^\infty\bra{\sphere}}\\
 &\leq C \sum_{j>J_n} B^{-js}= O\bra{n^{-\frac{s}{d}}}=O\bra{n^{-\frac{s}{2s+d}}}.
\end{align*}
As far as the other term is concerned, following the same procedure as described in the proof of Theorem \ref{thm:main}, see also \cite{bkmpAoSb,dmg}, we obtain 
\begin{equation*}
\Ex\bra{\norm{\sum_{j=0}^{J_n}\sum_{k=1}^{K_j}\bra{\tau_j \betaest-\betac}\psi_{j,k}}_{L^{\infty}\bra{\sphere}}} \leq C \bra{Aa+Au+Ua+Uu},
\end{equation*}
where
\begin{align*}
Aa & \!=\!\sum_{j=1}^{J_n}\Ex\bra{\norm{\sum_{k=1}^{K_j}\bra{\tau_j \betaest-\betac}\psi_{j,k}\ind{\abs{\widehat{\Theta}_j^\infty}\geq \frac{B^{dj}}{n^{\frac{1}{2}}}} \ind{\Theta_j^{\infty}\geq \frac{B^{dj}}{2n^{\frac{1}{2}}}}}_{L^{\infty}\bra{\sphere}} \!
	}\\
Au & \!=\!\sum_{j=1}^{J_n}\Ex\bra{\norm{\sum_{k=1}^{K_j}\bra{\tau_j \betaest-\betac}\psi_{j,k}\ind{\abs{\widehat{\Theta}_j^\infty}\geq \frac{B^{dj}}{n^{\frac{1}{2}}}} \ind{\Theta_j^{\infty}< \frac{B^{dj}}{2n^{\frac{1}{2}}}}}_{L^{\infty}\bra{\sphere}} \!}\\
Ua & \!=\!\sum_{j=1}^{J_n}\Ex\bra{\norm{\sum_{k=1}^{K_j}\betac\psi_{j,k}\ind{\abs{\widehat{\Theta}_j^\infty}\leq \frac{B^{dj}}{n^{\frac{1}{2}}}} \ind{\Theta_j^{\infty}\geq \frac{2B^{dj}}{n^{\frac{1}{2}}}}}_{L^{\infty}\bra{\sphere}} \!}\\
Uu & \!=\!\sum_{j=1}^{J_n}\Ex\bra{\norm{\sum_{k=1}^{K_j}\betac\psi_{j,k}\ind{\abs{\widehat{\Theta}_j^\infty}< \frac{B^{dj}}{n^{\frac{1}{2}}}} \ind{\Theta_j^{\infty}< \frac{B^{dj}}{2n^{\frac{1}{2}}}}}_{L^{\infty}\bra{\sphere}} \!}
\end{align*}
Now $\Theta_j^\infty\geq B^{dj}/n^{\frac{1}{2}}$ implies $j\leq J_s$ (see Section \ref{sub:thres}) and in view of Eq.s~\eqref{eqn:Lpnorm} and \eqref{eqn:boundsuppa} with $p=\infty$, we get 
\begin{align*}
Aa \leq &C\sum_{j=1}^{J_s} B^{\frac{d}{2}j}\Ex\bra{\sup_{k=1,\ldots,K_j}|{\hat \beta}_{j,k} - \beta_{j,k}|}\\
\leq & CB^{\frac{d}{2}J_s}\bra{J_s+1} n^{-\frac{1}{2}}\\
= &O\bra{n^{-\frac{s}{2s+d}}}
\end{align*}
Consider now $Au$. Using $J_s$, we split this term into $Au=Au_1+Au_2$, as in the proof of Theorem \ref{thm:main}. Trivially, we get 
\begin{equation*}
Au_1 = O\bra{n^{-\frac{s}{2s+d}}}.
\end{equation*}
On the other hand, using Eq.~\eqref{eqn:boundsuppa} and Proposition \ref{prop:threspuno}, we get
\begin{align*}
	Au 
	\leq & \sum_{j=0}^{J_n}B^{\frac{d}{2}j} \Ex\bra{\sup_{k=1,\ldots,K_j}|{\hat \beta}_{j,k} - \beta_{j,k}|^2}^{\frac{1}{2}}\Pr\bra{\abs{\widehat{\Theta}_j^{\infty}-\Theta_j^\infty}\geq \frac{B^{dj}}{2n^{\frac{1}{2}}}}^{\frac{1}{2}}\\
	\leq& \sum_{j=0}^{J_n}B^{\frac{d}{2}j}\bra{j+1}n^{-\frac{1}{2}} B^{-\frac{d}{2}j}\\
	=&O\bra{J_n n^{-\frac{1}{2}}}=o\bra{n^{-\frac{s}{2s+d}}}
	\end{align*}
As far as $Ua$ is concerned, again $\Theta_j^\infty\geq B^{dj}/n^{\frac{1}{2}}$ implies $j\leq J_s$ (see Section \ref{sub:thres}), so that 
\begin{align*}
Ua &\leq\sum_{j=1}^{J_s} B^{j\frac{d}{2}}\norm{\sum_{k=1}^{K_j}\betac\psi_{j,k}}_{L^{\infty}\bra{\sphere}}\Pr\bra{ \abs{\widehat{\Theta}_j^\infty- \Theta_j^\infty}\geq \frac{B^{dj}}{n^{\frac{1}{2}}} }\\
&\leq\sum_{j=1}^{J_s} B^{j\frac{d}{2}}M\Pr\bra{ \abs{\widehat{\Theta}_j^\infty- \Theta_j^\infty}\geq \frac{B^{dj}}{n^{\frac{1}{2}}} }\\
&\leq \sum_{j=1}^{J_s} B^{j\frac{d}{2}}M \bra{\frac{B^{dj}}{n^{\frac{1}{2}}} }^{-2}\Ex\bra{ \abs{\widehat{\Theta}_j^\infty- \Theta_j^\infty}^2} \\
&\leq J_s n^{-\frac{1}{2}},
\end{align*}
where we used Eq.~\eqref{eqn:boundsuppa} and Proposition \ref{prop:threspuno}. 
Finally, we have that
\begin{align*}
Uu & \leq \sum_{j=1}^{J_n}\Ex\bra{\norm{\sum_{k=1}^{K_j}\betac\psi_{j,k}\ind{\abs{\widehat{\Theta}_j^\infty}< \frac{B^{dj}}{n^{\frac{1}{2}}}} \ind{\Theta_j^{\infty}< \frac{B^{dj}}{2n^{\frac{1}{2}}}}}_{L^{\infty}\bra{\sphere}}}\\
& \leq Uu_1+Uu_2,
\end{align*}
where 
\begin{align*}
Uu_1 &\leq \sum_{j=1}^{J_s}\norm{\sum_{k=1}^{K_j}\betac\psi_{j,k}}_{L^{\infty}\bra{\sphere}} \ind{\Theta_j^{\infty}< \frac{B^{dj}}{2n^{\frac{1}{2}}}}\\
&\leq \sum_{j=1}^{J_s}B^{\frac{3d}{2}j}n^{-\frac{1}{2}}
=O\bra{B^{\frac{3d}{2}J_s}n^{-\frac{1}{2}}}
\end{align*}
and
\begin{align*}
Uu_2 =\sum_{j>J_s} B^{\frac{d}{2}j}\sup_{k=1,\ldots,K_j}\abs{\betac}.
\end{align*}
Note that 
\begin{align*}
B^{\frac{3d}{2}J_s}n^{-\frac{1}{2}}=n^{\frac{d-s}{2s+d}},
\end{align*}
as claimed.
\end{proof}

\subsection{Proofs of the auxiliary results}
The proof of Lemma \ref{lemma:needloc} can be viewed as a generalization of the proof of Lemma 5.1 in Durastanti et al. \cite{dmp}. 
\begin{proof}[Proof of Lemma \ref{lemma:needloc}]
Using the needlets localization property given in Eq.~\eqref{eqn:needloc}, we have that
\begin{equation*}
\int_{\sphere} \prod_{i=1}^q \psi_{j,k_i}\bra{x}dx \leq C_\eta B^{jdq} \int_{\sphere} \prod_{i=1}^q \frac{1}{\bbra{1+B^{jq}d\bra{x,\xi_{j,k_i}}}^\eta}.
\end{equation*}
Let $S_1 = \bbra{x \in \sphere:d\bra{x,\xi_{j,k_1}}\geq \Delta/2 }$, so that $\sphere \subseteq S_1 \cup \overline{S_1}$. Therefore,
\begin{align*}
\int_{\sphere} \prod_{i=1}^q \frac{1}{\bbra{1+B^{jq}d\bra{x,\xi_{j,k_i}}}^\eta}\leq &\int_{S_1} \prod_{i=1}^q \frac{1}{\bbra{1+B^{jq}d\bra{x,\xi_{j,k_i}}}^\eta}\\ &+
\int_{\overline{S_1}} \prod_{i=1}^q \frac{1}{\bbra{1+B^{jd}d\bra{x,\xi_{j,k_i}}}^\eta}.
\end{align*}
From the definition of $\overline{S_1}$ and following Lemma 5.1 in Durastanti, Marinucci and Peccati \cite{dmp}, it follows that
\begin{align*}
\int_{\overline{S_1}} \prod_{i=1}^q \frac{1}{\bbra{1+B^{jd}d\bra{x,\xi_{j,k_i}}}^\eta} \leq&
\frac{2^{\eta\bra{q-1}}}{\bra{1+B^{jd}\Delta}^{\eta\bra{d-1}}}
\int_{\overline{S_1}}\frac{1}{\bbra{1+B^{jd}d\bra{x,\xi_{j,k_1}}}^\eta}dx\\ 
\leq&
\frac{2^{\eta\bra{q-1}}}{\bra{1+B^{jd}\Delta}^{\eta\bra{d-1}}}
\int_{\sphere}\frac{1}{\bbra{1+B^{jd}d\bra{x,\xi_{j,k_1}}}^\eta}dx\\
=&\frac{\bra{2\pi}^{d-1} 2^{\eta\bra{q-1}}}{\bra{1+B^{jd}\Delta}^{\eta\bra{d-1}}}
\int_{0}^{\pi}\frac{\sin \theta}{\bra{1+B^{jd}\theta}^\eta}d\theta \\
\leq&\frac{\bra{2\pi}^{d-1} 2^{\eta\bra{q-1}}B^{-dj}}{\bra{1+B^{jd}\Delta}^{\eta\bra{d-1}}}
\int_{0}^{\infty}\frac{y}{\bra{1+y}^\eta}dy \\
\leq&C_\eta^\prime\frac{\bra{2\pi}^{d-1} 2^{\eta\bra{q-1}}B^{-dj}}{\bra{1+B^{jd}\Delta}^{\eta\bra{d-1}}}.
\end{align*}
On the other hand, 
\begin{align*}
\int_{S_1} \prod_{i=1}^q \frac{1}{\bbra{1+B^{jd}d\bra{x,\xi_{j,k_i}}}^\eta} \leq&
\frac{2^{\eta}}{\bra{1\!+\!B^{jd}\Delta}^{\eta}}
\int_{S_1}\prod_{i=1}^{q-1}\!\frac{dx}{\bbra{1+B^{jd}d\bra{x,\xi_{j,k_{i+1}}}}^\eta}. 
\end{align*}
Let $S_2=\bbra{x \in S_1: d\bra{x,\xi_{j,k_2}}\geq \Delta/2}$. Then,
\begin{equation*}
S_1 \subseteq S_2 \cup \overline{S_2} \quad \text{ and } \quad \sphere \subseteq S_2\cup\overline{S_2}\cup\overline{S_1}.
\end{equation*}
 As far as $\overline{S_2}$ is concerned, we apply the same chain of inequalities as those used for $\overline{S_1}$. The integral over $S_2$ can be bound by the factor $2^{\eta}\bra{1+B^{jd}\Delta}^{-\eta}$ multiplied by the integral of the product of $q-2$ localization bounds of the needlets. By re-iterating the procedure, we obtain, a set of nested $S_g=\bbra{x \in S_{g-1}, d\bra{x,\xi_{j,k_g}}\geq \Delta/2}$, $g=1,\ldots,q$ so that $\sphere \subseteq S_q \cup \bigcup_{g=1}^q \overline{S_g}$, which yields the claimed result.
\end{proof}

\noindent The proof of Proposition \ref{prop:betac} is a simple modification of Proposition 6 in Durastanti et al. \cite{dmg} concerning complex random spin needlet coefficients. Many technical details are omitted for the sake of brevity. 
\begin{proof}[Proof of Proposition \ref{prop:betac}] For $p\leq 2$ we apply the classical convexity inequality such that for  a set of independent centered random variables $\bbra{Z_i}$ with finite $p$-th absolute moment, 
\begin{equation*}
\Ex\bra{\abs{\sum_{i=1}^n Z_i}^p}\leq \bbra{\Ex\bra{\abs{\sum_{i=1}^nZ_i}^2}}^{{p/2}}.
\end{equation*}
For $p>2$, we apply the Rosenthal inequality (see for instance H\"ardle et al. \cite{WASA}), that is, there exists a constant $c_p>0$ such that
\begin{equation*}
\Ex\bra{\abs{\sum_{i=1}^n Z_i}^p}\leq c_p \sbra{\sum_{i=1}^n\Ex\bra{\abs{Z_i}^p}+\bbra{\sum_{i=1}^n \Ex\bra{Z_i^2}}^{{p/2}}} 
\end{equation*}
On the other hand, since $B^{dj}\leq n$, we have that 
\begin{align*}
\Ex \bbra{\abs{\bra{f\bra{X}+\varepsilon}\needlet{X}-\betac}^p}\leq & 2^{p-1} \left[ \Ex\bbra{\abs{f\bra{X}\needlet{X}-\betac}^p}\right.\\& \left.+\Ex\bbra{\abs{\varepsilon\needlet{X}}^p} \right] \\
 \leq & c_p^\prime \bbra{M^p + \Ex\bra{\abs{\varepsilon}^p} }\norm{\psi_{j,k}}_\Lp^p\\
 \leq & c_p^{\prime \prime} B^{jd\bra{{p/2}-1}} \leq c_p^{\prime \prime \prime} n^{{p/2}-1}
\end{align*}  
Hence,
\begin{align*}
\Ex\bra{|{\hat \beta}_{j,k} - \beta_{j,k}|^p} \leq &\tilde{c}_p\bra{\frac{n^{{p/2}-1}}{n^{p-1}}+n^{-{p/2}}} =\tilde{c}_p n^{-{p/2}}
\end{align*}
\end{proof}
\noindent The proof of Proposition \ref{prop:thres} can be considered as the counterpart in the needlet framework of the proof of Lemma 2 in Kerkyacharian et al. \cite{kpt96}. 
\begin{remark} Any element $c_\ll$ can be decomposed as the sum of integers $c_{{i_1},\ldots,{i_\ell};\ll}$, where the $\ll$-dimensional vector $\bbra{{i_1},\ldots,{i_\ell}}\subset \bbra{1,\ldots,m}$ specifies the spherical needlets involved in each configuration (of size $\ll$) given by
\begin{equation*}\Ex\sbra{\bbra{f\bra{X}+\varepsilon}^\ll\psi_{j,k_{i_1}}\bra{X}\ldots,\psi_{j,k_{i_\ll}}\bra{X}}.
\end{equation*} 
The notation $\sbra{i_1,\ldots,i_\ll}$ denotes the set of all the possible combinations of $\bbra{i_1,\ldots,i_\ll}$ such that 
\begin{eqnarray*}
\sum_{\sbra{i_1,\ldots,i_\ll}}c_{i_1,\ldots,i_\ll;\ll}=c_\ll.
\end{eqnarray*}
\end{remark}
\begin{proof}[Proof of Proposition \ref{prop:thres}]
Note that 
\begin{align}\label{eqn:eqn0}
\Ex\sbra{\bbra{\Thetaest}^m} &=\binom{n}{p}^{-m} \sum_{\upsilon_1,\dots,\upsilon_m \in \Sigma_p} \Ex\bbra{\sum_{k{_1},\ldots,k_m} \prod_{\ll=1}^m\Psi_{j,k_{\ll}}^{\otimes p} \bra{X_{\upsilon_\ll},\varepsilon_{\upsilon_\ll}} }
\end{align}
where
\begin{align*}\label{eqn:eqn1}
&\Ex\bbra{\sum_{k{_1},\ldots,k_m} \prod_{\ll=1}^m\Psi_{j,k_{\ll}}^{\otimes p} \bra{X_{\upsilon_\ll},\varepsilon_{\upsilon_\ll}} }\\
&\quad \quad \quad \quad \quad \quad = \sum_{k_1,\ldots,k_m}\!\prod_{\ll=1}^m\!\prod_{\left[i_1,\ldots,i_\ll\right]}\!\Ex\sbra{\bbra{f\bra{X}+\varepsilon}^\ll \prod_{h=1}^\ll \psi_{j,k_{i_h}}\bra{X}}^{c_{i_1,\ldots,i_{\ll};\ll}}\notag\\
&\quad \quad \quad \quad \quad \quad =\sum_{k_1,\ldots,k_m}\!\prod_{h=1}^m \Ex\bbra{f\bra{X} \prod_{h=1}^\ll \psi_{j,k_{h}}\bra{X}}^{c_{h;1}}\notag\\ 
&\quad \quad \quad \quad \quad \quad \quad \prod_{\ll=2}^m\!\prod_{\left[i_1,\ldots,i_\ll\right]}\Ex\sbra{\bbra{f\bra{X}+\varepsilon}^\ll \prod_{h=1}^\ll \psi_{j,k_{i_h}}\bra{X}}^{c_{i_1,\ldots,i_{\ll};\ll}}.\notag
\end{align*}  
Using Eq.~\eqref{eqn:boundf} and the independence of the noise $\varepsilon$, for any $\ll \geq 2$ we have that 
\begin{align*}
\prod_{\left[i_1,\ldots,i_\ll\right]}\!&\Ex\sbra{\bbra{f\bra{X}+\varepsilon}^\ll \prod_{h=1}^\ll \psi_{j,k_{i_h}}\bra{X}}^{c_{i_1,\ldots,i_{\ll};\ll}}\! \! \! \!\!\! \! \! \!\!\! \! \! \!\!\leq C_{M,p,\ll} \Ex \bbra{\prod_{h=1}^\ll \psi_{j,k_h}\bra{X}}^{c_{i_1,\ldots,i_\ll,\ll} }\! \! \! ,
\end{align*}
where $C_{M,p,\ll}=2^{p-1}\bbra{M^\ll+\Ex\bra{\varepsilon^\ll}}$. In view of Lemma \ref{lemma:needloc}, we obtain
\begin{align*}
\prod_{\ll=2}^m\prod_{\left[i_1,\ldots,i_\ll\right]}\Ex \bbra{\prod_{h=1}^\ll \psi_{j,k_h}\bra{X}}^{c_{i_1,\ldots,i_\ll,\ll} } \! \! \! \!\!\leq&  C^\prime\prod_{\ll=2}^m\prod_{\left[i_1,\ldots,i_\ll\right]}\norm{\psi_{j,k}}_{L^\ell \bra{\sphere}}^{\ll c_{i_1,\ldots,\i_\ll;\ll}} \\
&\quad \quad \quad \quad  \ind{k\!=\!k_{i_1}\!=\!\ldots\!=\!k_{i_\ll}}\\
\leq& C^\prime\prod_{\ll=2}^m\norm{\psi_{j,k}}_{L^\ell \bra{\sphere}}^{\ll \sum_{\left[i_1,\ldots,i_\ll\right]}c_{i_1,\ldots,\i_\ll;\ll}} \Delta\bra{k,\ell}\\
\leq &C^{\prime \prime} B^{jd\bra{\sum_{\ll=2}^m \frac{\ll c_\ll}{2}-\sum_{\ll=2}^m c_\ll}}\Delta\bra{k,m}\\
= &C^{\prime \prime} B^{jd\bra{\sum_{\ll=1}^m \frac{\ll c_\ll}{2}-\sum_{\ll=1}^m c_\ll}}B^{jd\frac{c_1}{2}}\Delta\bra{k,m},
\end{align*}
Note that $\prod_{\left[i_1,\ldots,i_\ll\right]}\ind{k\!=\!k_{i_1}\!=\!\ldots\!=\!k_{i_\ll}}$ implies that al least $\ll$ $k_h$ indexes are equal. 
Thus, using Eq.~\eqref{eqn:besovmin}, we obtain
\begin{align*}
&\sum_{k_1,\ldots,k_m}\prod_{h=1}^m\Ex\bbra{f\bra{X}\psi_{j,k_h}\bra{X}}^{c_{h,1}}=\sum_{k}\beta_{j,k}^{\sum_{h=1}^m c_{h,1}}\\
&\quad \quad \quad \quad \quad\quad \quad \quad  \leq C  B^{-j\sum_{h=1}^m c_{h,1}\bbra{s+d\bra{\frac{1}{2}-\frac{1}{p}}}} B^{jd\bbra{1-\frac{\min\bra{p,\sum_{h=1}^m c_{h;1}}}{p}}}\\
&\quad \quad \quad \quad \quad \quad\quad \quad =C  B^{-jc_{1}\bbra{s+d\bra{\frac{1}{2}-\frac{1}{p}}}} B^{jd\bra{1-\gamma}}
\end{align*}
where $\gamma =\min\bra{p,c_{1}}/p$. Hence,
\begin{align*}
\prod_{\left[i_1,\ldots,i_\ll\right]}\!\!\Ex\sbra{\bbra{f\bra{X}+\varepsilon}^\ll \prod_{h=1}^\ll \psi_{j,k_{i_h}}\bra{X}}^{c_{i_1,\ldots,i_{\ll};\ll}}\!\leq& C^{\prime \prime \prime}B^{jd\bra{\frac{mp}{2}-\sum_{\ll=1}^m c_\ll}} \\ &\quad B^{-jc_{1}\bra{s-\frac{d}{p}}} B^{jd\bra{1-\gamma}}
\end{align*}
Finally, for any fixed configuration $c_1,\ldots,c_m$, the number of possible combinations is bounded by   
\begin{equation*}
C^\star\binom{n}{c_m}\binom{n-c_m}{c_{m-1}}\ldots\binom{n-\sum_{\ell=2}^m{c_\ell}}{c_1},
\end{equation*}
where $C^\star$ denotes the possible choices of $\bbra{c_{i_1,\ldots,i_\ll;\ll}}$ for any $\ll$ and does not depend on $n$. 
In view of the Stirling approximation, $\binom{n}{p}\approx n^p$, the number of possible combinations is bounded by $Cn^{-\sum_{\ll=1}^{m}c_\ll}$. Using the aforementioned results, Eq.~\eqref{eqn:eqn0} is bounded by 
 \begin{align*}
\Ex\sbra{\bbra{\Thetaest}^m} \leq  \tilde{C}_1 \!\!\!\!\sum_{\bra{c_1,\ldots, c_m} \in \Gamma_{m,p}}\!\!\!\!\bra{\frac{B^{jd}}{n}}\!^{\bra{\frac{mp}{2}-\sum_{\ell=1}^m\! c_\ell}} \frac{B^{-j\bbra{c_1 \bra{s-\frac{d}{p}}-d\bra{1-\gamma}}}}{n^{\frac{mp}{2}}}, 
\end{align*}
as claimed.
\end{proof}
The proof of Proposition \ref{prop:threscent} can be viewed as the counterpart of the proof of Lemma 3 in Kerkyacharian et al. \cite{kpt96} in the needlet framework. 
\begin{proof}[Proof of Proposition \ref{prop:threscent}]
Following Kerkyacharian et al. \cite{kpt96}, note that
\begin{equation*}
\prod_{i=1}^p x_i -\beta^p = \sum_{h=1}^p \beta^{p-h} \sum_{1\leq t_1 < \ldots, < t_h \leq p}\prod_ {i=1}^h \bra{x_{t_i}-\beta}.
\end{equation*}
Let 
\begin{equation*}
\tilde{\Psi}_{j,k}\bra{X,\varepsilon}:=\bbra{f\bra{x}+\varepsilon}\needlet{x}-\betac. 
\end{equation*} 
so that 
\begin{equation}\label{eqn:deggy}
\Ex\bbra{\tilde{\Psi}_{j,k}\bra{X,\varepsilon}}=0.
\end{equation}
We therefore obtain
\begin{equation*}
\Thetaest-\Theta_j\bra{p}= \binom{n}{p}^{-1}\sum_{k=1}^{K_j} \sum_{\upsilon \in \Sigma_p}  \sum_{h=1}^p \betac^{p-h}\sum_{\iota \subset \upsilon, \iota \in \Sigma_h}\tilde{\Psi}_{j,k}^{\otimes h}\bra{X_{\iota},\varepsilon_{\iota}},
\end{equation*}
and, reversing the order of integration, we have that
\begin{equation*}
\Thetaest-\Theta_j\bra{p}=\sum_{k=1}^{K_j} \sum_{h=1}^p \frac{\binom{n-h}{p-h}}{\binom{n}{p}}  \betac^{p-h}\sum_{\iota \in \Sigma_h}\tilde{\Psi}_{j,k}^{\otimes h}\bra{X_{\iota},\varepsilon_{\iota}},
\end{equation*}
Hence, we can rewrite
\begin{align*}
\Ex\bbra{\bra{\Thetaest -\Theta_j\bra{p}}^m} \leq & p^{m-1} \sum_{h=1}^{p} \bbra{\frac{\binom{n-h}{p-h}}{\binom{n}{p}}}^m  \sum_{k_1,\ldots,k_m}\abs{\prod_{\ll=1}^m\beta_{j,k_\ll}}^{p-h}\\
&\sum_{\iota_{1},\ldots,\iota_{m}\in\Sigma_h}\Ex\bbra{\prod_{\ll=1}^m \tilde{\Psi}_{j,k_\ll}^{\otimes h}\bra{X_{\iota_{\ll}},\varepsilon_{\iota_\ll}} }.
\end{align*}
Similar to the proof of Proposition \ref{prop:thres}, we fix a configuration of indexes ${\iota_1,\ldots,\iota_m \in \Sigma_{h}}$, corresponding to the set of coefficients $\bbra{c_1,c\ldots,c_h}$. Because, in this case, the considered U-statistic is degenerate, we discard all the combinations with $c_1 \neq 0$, in view of \eqref{eqn:deggy}. On the other hand, following Lemma \ref{lemma:needloc}, we have that
\begin{align}\label{eqn:esponente}
\notag \sum_{\iota_{1},\ldots,\iota_{m}\in\Sigma_h}\Ex\bbra{\tilde{\Psi}_{j,k_1}^{\otimes h}\bra{X_{1},\varepsilon_1}\ldots\tilde{\Psi}_{j,k_h}^{\otimes h}\bra{X_{h},\varepsilon_h} }& \leq C n^{\sum_{\ll=1}^h c_\ll} B^{dj\bbra{\sum_{\ll=2}^h \bra{\frac{\ll}{2}-1}c_\ll}}\\
&=Cn^{\sum_{\ll=1}^h c_\ll}  B^{dj\bra{\frac{mh}{2}-\sum_{\ll=2}^m c_\ll}}.
\end{align}
Furthermore, $mh=\sum_{\ll=2}^m \ll c_\ll > 2\sum_{\ll=2}c_\ll$ implies that the exponent in
the last term of \eqref{eqn:esponente} is positive, so that 
\begin{equation*}
\sum_{\iota_{1},\ldots,\iota_{m}\in\Sigma_h}\Ex\bbra{\tilde{\Psi}_{j,k_1}^{\otimes h}\bra{X_{1},\varepsilon_1}\ldots\tilde{\Psi}_{j,k_h}^{\otimes h}\bra{X_{h},\varepsilon_h} } \leq C n^{\frac{mh}{2}},
\end{equation*}
because $B^{dj}\leq n$. Finally, using the Stirling approximation and Eq.~\eqref{eqn:besovmin} we have that
\begin{align*}
\Ex\sbra{\bbra{\Thetaest -\Theta_j\bra{p}}^m} \leq & C^\prime \sum_{h=1}^p \frac{n^{m\bra{p-h}}}{n^{mp}} \bra{\sum_k \abs{\betac}^{p-h} }^m n^{\frac{mh}{2}}\\
\leq &  C^{\prime \prime}\sum_{h=1}^p n^{-\frac{mh}{2}} \bbra{B^{jd\frac{h}{p}}\bra{\sum_k \abs{\betac}^{p}}^{1-\frac{h}{p}} }^m \\
\leq &  \tilde{C}_2 \sum_{h=1}^p \bra{\frac{B^{jd}}{n^{{p/2}}}}^{\frac{mh}{p}}\bbra{ \Theta_j\bra{p}}^{\bra{1-\frac{h}{p}}m},  
\end{align*}
as claimed.
\end{proof}
The proof of Lemma \ref{lemma:central} is the counterpart of the proof of Lemma 6 in Kerkyacharian et al. \cite{kpt96} in the needlet framework. 
\begin{proof}[Proof of Lemma \ref{lemma:central}]
	Following Kerkyacharian et al. \cite{kpt96} and the results obtained in the previous proof, we have that
	\begin{align*}
		\hat{\Theta}_j\bra{p_0}-\Theta_j\bra{p_0}= &\binom{n}{p_0}^{-1}\sum_{k=1}^{K_j} \sum_{\upsilon \in \Sigma_{p_0}}  \sum_{h=1}^{p_0^\prime} \betac^{p_0^\prime-h}\left\{\sum_{\iota \subset \upsilon, \iota \in \Sigma_h}\tilde{\Psi}_{j,k}^{\otimes h}\bra{X_{\iota},\varepsilon_{\iota}}\right.\\
		 &+\frac{1}{n^{\frac{1}{2}}}\sum_{\iota \subset \upsilon, \iota \in \Sigma_{h-1}}\tilde{\Psi}_{j,k}^{\otimes \bra{h-1}}\bra{X_{\iota},\varepsilon_{\iota}} \\
		& \left.+\frac{1}{n}\sum_{\iota \subset \upsilon, \iota \in \Sigma_{h-2}}\tilde{\Psi}_{j,k}^{\otimes \bra{h-2}}\bra{X_{\iota},\varepsilon_{\iota}} \right\},		
	\end{align*}
with the convention
\begin{align*}
\sum_{\iota \subset \upsilon, \iota \in \Sigma_0}\tilde{\Psi}_{j,k}^{\otimes \bra{0}}\bra{X_{\iota},\varepsilon_{\iota}}=2; \ \sum_{\iota \subset \upsilon, \iota \in \Sigma_{h}}\tilde{\Psi}_{j,k}^{\otimes \bra{h}}\bra{X_{\iota},\varepsilon_{\iota}}=0 \ \text{for }h<0. 
\end{align*}
Reversing the order of integration and applying an analogous procedure to the one used in the proof of Proposition \ref{prop:threscent}, we achieve the claimed result.
\end{proof}	
Proposition \ref{prop:threspuno} is proved by using the general properties of the needlets.
\begin{proof}[Proof of Proposition \ref{prop:threspuno}]
It is easy to see that
	\begin{align*}
	\Ex\!\bbra{\!\bra{\widehat{\Theta}_j^\infty}^2\!} \!&=\! \frac{1}{n^2}\!\sum_{k=1}^{K_j}\sum_{i=1}^n\Ex\!\sbra{\!\bbra{\needlet{X_i}Y_i}^2}\!\!+\!\!\sbra{\!\frac{1}{n}\sum_{k=1}^{K_j}\sum_{i=1}^{n}\Ex\!\bbra{\!
			\needlet{X_i}Y_i}\!}^2 \\
	&\leq \frac{B^dj}{n}\bra{M^2+\sigma_\varepsilon^2}\norm{\psi_{j,k}}_{\Ltwo} + \bra{\Theta_j^\infty}^2,
	\end{align*}
as claimed.
\end{proof}

\begin{acknow*}
{\rm The author wishes to thank M. Konstantinou, A. Ortiz, A. Renzi and N. Turchi for precious discussions and hints. Furthermore, the author wishes to acknowledge the Associate Editor, the referees and the Editor-in-Chief for the insightful remarks and suggestions which led to a substantial improvement of this work.}  
\end{acknow*}


\end{document}